\documentclass[a4paper,10pt]{article}

\usepackage[
  left=1.5cm,
  right=1.5cm,
  top=2cm,
  bottom=2cm
]{geometry}
\usepackage{amsmath, amssymb, amsthm}
\usepackage{mathtools}
\usepackage{graphicx}
\usepackage{caption}
\usepackage{subcaption}
\usepackage{enumitem}
\usepackage{tikz}
\usepackage{tikz-cd}
\usepackage[all]{xy}
\usepackage[dvipsnames]{xcolor}
\usepackage{longtable}
\usepackage{booktabs}
\usepackage{multirow}
\usepackage[figuresright]{rotating}
\usepackage{thmtools}
\usepackage{thm-restate}
\usepackage{hyperref}
\usepackage[nameinlink]{cleveref}

\hypersetup{hypertexnames=false}

\newcommand{\keywords}[1]{\par\noindent\textbf{Keywords:} #1}

\newtheorem{theorem}{Theorem}[section]
\newtheorem{definition}[theorem]{Definition}
\newtheorem{lemma}[theorem]{Lemma}
\newtheorem{proposition}[theorem]{Proposition}
\newtheorem{corollary}[theorem]{Corollary}
\newtheorem{remark}[theorem]{Remark}
\newtheorem{example}[theorem]{Example}

\crefname{theorem}{theorem}{theorems}
\crefname{lemma}{lemma}{lemmas}
\crefname{proposition}{proposition}{propositions}
\crefname{corollary}{corollary}{corollaries}
\crefname{definition}{definition}{definitions}
\crefname{example}{example}{examples}

\newcommand{\im}{\operatorname{Im}}
\newcommand{\bba}{\mathbb{A}}
\newcommand{\bbd}{\mathbb{D}}
\newcommand{\bbf}{\mathbb{F}}
\newcommand{\bbr}{\mathbb{R}}
\newcommand{\bbz}{\mathbb{Z}}
\newcommand{\calc}{\mathcal{C}}

\newcommand{\calfm}{\mathcal{F\!M}}
\newcommand{\calfp}{\mathcal{FP}}
\newcommand{\calfc}{\mathcal{FC}}
\newcommand{\calp}{\mathcal{P}}

\DeclareMathOperator{\fdl}{FDL}
\DeclareMathOperator{\rank}{Rank}
\DeclareMathOperator{\cut}{Cut}
\DeclareMathOperator{\sub}{Sub}
\DeclareMathOperator{\solv}{Solv}

\newcommand{\SpCpx}{\textrm{SpCpx}}
\newcommand{\E}{E^\Delta}
\newcommand{\mdh}{M^{\Delta,H}}
\newcommand{\mhd}{M^{H,\Delta}}

\DeclareMathOperator{\Mo}{M}
\DeclareMathOperator{\No}{N}
\DeclareMathOperator{\Co}{C}
\DeclareMathOperator{\Zo}{Z}
\DeclareMathOperator{\Bo}{B}
\DeclareMathOperator{\Ho}{H}

\definecolor{myblue}{RGB}{100, 200, 200}
\definecolor{myred}{RGB}{150, 0, 0}

\begin{document}

\title{L-fuzzy simplicial homology}

\author{Javier Perera-Lago\(^1\) \and Alvaro Torras-Casas\(^1\) \and Rocio Gonzalez-Diaz\(^1\)}
\date{\(^1\)Universidad de Sevilla}

\maketitle

\begin{abstract}
Simplicial homology is a classical tool that assigns a sequence of modules to a simplicial complex, providing invariants for the study of its topological properties. In this article, we introduce the notion of \(L\)-fuzzy simplicial homology, a generalization of simplicial homology for \(L\)-fuzzy subcomplexes, in which each simplex is assigned a value from a completely distributive lattice \(L\). We present its definition and main properties and describe methods to compute its structure. In addition, we interpret filtrations over a poset and chromatic datasets in this setting, opening a door to further applications in topological data analysis.
\end{abstract}



\section{Introduction}
\label{sec:introduction}

\emph{Topological Data Analysis} (TDA) is a research field that complements traditional data analysis by applying theory and algorithms from computational topology. A significant part of TDA is devoted to the search for invariant descriptors of datasets (represented as point clouds in Euclidean space \(\bbr^n\)). By constructing a simplicial complex from a dataset, one can compute its homology modules \(\Ho_d\) (for \(d\geq 0\)) which capture information about connected components, cycles, cavities, and higher-dimensional features. 
Simplicial complexes are also central in TDA for dimensionality reduction and visualization. 
In particular, UMAP \cite{mcinnes2018umap} and IsUMap \cite{barth2025fuzzy} leverage the concept of \emph{fuzzy simplicial sets} to compute embeddings in \(\bbr^2\) or \(\bbr^3\).
Motivated by these developments, we address the following question: if fuzzy simplicial sets can be used for data visualization, can they also be used for data description? More precisely, is it possible to define and compute a meaningful notion of homology for a fuzzy simplicial set?

A review of the literature shows that existing studies on fuzzy simplicial sets mainly focus on dimensionality reduction \cite{spivak2009metric,mcinnes2018umap,barth2025fuzzy,keck2025probabilistic,reyes2025assessing}, often from a categorical perspective. Other works such as \cite{wang1985singular} compute singular homology for fuzzy topological spaces. However, the resulting invariants are abelian groups that are difficult to compute in practice. 
A different research line considers weighted simplicial complexes, where simplices have numerical weights encoding their relevance within the dataset \cite{dawson1990homology,ren2018weighted}. Although this approach enriches the homology information with the weights, the result is still an abelian group or module.
In contrast, the goal of this paper is to develop a homology theory intrinsically adapted to fuzzy simplicial complexes, yielding fuzzy invariants, better suited to represent the inherent imprecision of the dataset.

Besides, in this paper, we focus on fuzzy simplicial complexes rather than fuzzy simplicial sets 
to develop a framework that remains closer to the classical combinatorial setting. Furthermore, by replacing the classical interval \([0,1]\) of fuzzy membership values with a completely distributive lattice \((L,\leq)\), we extend our approach to \(L\)-fuzzy subcomplexes. Thus, we propose a method to define and effectively compute \(L\)-fuzzy simplicial homology, 
that is, a family of \(L\)-fuzzy submodules \(\eta_d\) (for \(d\geq 0\)) that extends classical homology modules while preserving the underlying fuzzy information of the dataset. 
This framework allows us to encompass a broader range of real-world applications, including filtrations over a poset and \emph{chromatic datasets}, whose points are endowed with a ``color'' indicating a class or category. The analysis of chromatic datasets has recently been considered in TDA \cite{di2024chromatic,natarajan2024morse}, but we revisit the problem through the lens of \(L\)-fuzzy simplicial homology.

The paper is organized as follows. In \Cref{sec:cdl}, we introduce  completely distributive lattices. In \Cref{sec:lfuzzysubsets}, we present the main definitions and results concerning \(L\)-fuzzy subsets, where \((L,\leq)\) is a completely distributive lattice.  In \Cref{sec:lfuzzysubmodules}, we define \(L\)-fuzzy submodules and establish several properties that will be needed later. In \Cref{sec:lfuzzysubcomplexes}, we introduce \(L\)-fuzzy subcomplexes and discuss how they can be used to model relevant structures arising in applications. In \Cref{sec:fuzhom}, we begin by reviewing  simplicial homology and then combine the concepts developed in the previous sections to define \(L\)-fuzzy simplicial homology and study its fundamental properties. In \Cref{sec:computation} we present a method for the computation of \(L\)-fuzzy simplicial homology, valid for any choice of completely distributive lattice \(L\) and any coefficient ring. In \Cref{sec:example}, we illustrate the theory with a detailed example for a chromatic dataset, including all computations required to determine both the standard and the \(L\)-fuzzy simplicial homology. Finally, \Cref{sec:conclusions} summarizes the main results and outlines directions for future work.
The main notations used in the paper can be consulted in~\ref{sec:listofnotations}.

\section{Completely distributive lattices}
\label{sec:cdl}

In this section we follow \cite{davey2002introduction} to introduce completely distributive lattices, which are partially ordered sets (posets) endowed with additional structure. These lattices serve as the sets of membership values for \(L\)-fuzzy sets.

Let \((P,\leq)\) be a poset. If \(p_1 \leq p_2\) and \(p_1 \neq p_2\), we write \(p_1 < p_2\). Two elements \(p_1, p_2 \in P\) are said to be \emph{comparable} if \(p_1 \leq p_2\) or \(p_2 \leq p_1\), and \emph{incomparable} otherwise. If every pair of elements in \(P\) is comparable, then \(\leq\) is called a \emph{total order}. Given \(p \in P\) and a subset \(S \subseteq P\), we write \(S \leq p\) (respectively \(p \leq S\)) if \(s \leq p\) for every \(s \in S\) (respectively \(p \leq s\) for every \(s \in S\)). 

\begin{definition}[Completely distributive lattice]
A poset \((L,\leq)\) is called a \emph{completely distributive lattice} (CDL) if the following conditions hold:
\begin{itemize}
\item For every subset \(S \subseteq L\), there exists a unique element \(\bigvee S \in L\), called the \emph{join of \(S\)}, such that \(S \leq \bigvee S\), and for any \(\ell \in L\) with \(S \leq \ell\), one has \(\bigvee S \leq \ell\). If \(S = \{\ell_1,\dots,\ell_n\}\) is finite, \(\bigvee\! S\) can also be written as \(\ell_1 \vee \cdots \vee \ell_n\).
\item For every subset \(S \subseteq L\), there exists a unique element \(\bigwedge S \in L\), called the \emph{meet of \(S\)}, such that \(\bigwedge S \leq S\), and for any \(\ell \in L\) with \(\ell \leq S\), one has \(\ell \leq \bigwedge S\). If \(S = \{\ell_1,\dots,\ell_n\}\) is finite, \(\bigwedge\! S\) can also be written as \(\ell_1 \wedge \cdots \wedge \ell_n\).
\item For every doubly indexed family \(\{\ell_{ij} \in L \mid i \in I, j \in J\}\), arbitrary meets and joins distribute, that is,
\[
\bigwedge_{i \in I} \bigvee_{j \in J} \ell_{ij}
=
\bigvee_{f \in J^I} \bigwedge_{i \in I} \ell_{i f(i)}
\quad \text{and} \quad
\bigvee_{i \in I} \bigwedge_{j \in J} \ell_{ij}
=
\bigwedge_{f \in J^I} \bigvee_{i \in I} \ell_{i f(i)}.
\]
\end{itemize}
\end{definition}

Since a CDL \((L,\leq)\) admits joins and meets of arbitrary subsets, the elements \(\bigvee \emptyset\) and \(\bigwedge \emptyset\) exist. Any element \(\ell \in L\) trivially satisfies the conditions \(\emptyset \leq \ell\) and \(\ell \leq \emptyset\). Therefore, \(\bigvee \emptyset\) is the least element of \(L\), denoted \(0\), and \(\bigwedge \emptyset\) is the greatest element, denoted \(1\). Hence every CDL is bounded.

\begin{theorem}[\cite{davey2002introduction}]
\label{thm:lattices_laws}
Let \((L,\leq)\) be a CDL and let \(\ell_1,\ell_2,\ell_3 \in L\). Then, the following properties hold:
\begin{itemize}
\item (Connecting lemma) \(\ell_1 \leq \ell_2\) if and only if \(\ell_1 \vee \ell_2 = \ell_2\), equivalently if \(\ell_1 \wedge \ell_2 = \ell_1\).
\item (Idempotency) \(\ell_1 \vee \ell_1 = \ell_1\) and \(\ell_1 \wedge \ell_1 = \ell_1\).
\item (Commutativity) \(\ell_1 \vee \ell_2 = \ell_2 \vee \ell_1\) and \(\ell_1 \wedge \ell_2 = \ell_2 \wedge \ell_1\).
\item (Associativity) \((\ell_1 \vee \ell_2) \vee \ell_3 = \ell_1 \vee (\ell_2 \vee \ell_3)\) and \((\ell_1 \wedge \ell_2) \wedge \ell_3 = \ell_1 \wedge (\ell_2 \wedge \ell_3)\).
\item (Absorption) \(\ell_1 \vee (\ell_1 \wedge \ell_2) = \ell_1\) and \(\ell_1 \wedge (\ell_1 \vee \ell_2) = \ell_1\).
\end{itemize}
\end{theorem}

\begin{example}
\label{exa:totally}
Every bounded totally ordered set, such as \(([0,1],\leq)\), is a CDL where \(\bigwedge S = \inf S\) and \(\bigvee S = \sup S\).
\end{example}

\begin{definition}[Free distributive lattice]
\label{def:fdl}
Let \(S=\{x_1,\ldots,x_n\}\) be a set with no order relations imposed. Consider the set \(M=\{\bigwedge T \mid T \subseteq S\}\) of all formal finite meets of elements of \(S\), including \(\bigwedge \emptyset = 1\). Define also the set \(L=\{\bigvee T \mid T \subseteq M\}\) of all formal joins of elements of \(M\), including \(\bigvee \emptyset = 0\). We induce an order relation \(\leq\) on \(L\) by declaring \(\ell_1 \leq \ell_2\) if and only if the identity \(\ell_1 \vee \ell_2 = \ell_2\) can be derived using the laws from \Cref{thm:lattices_laws} together with distributivity. The resulting poset \((L,\leq)\) is the \emph{free distributive lattice} \(\fdl(x_1,\ldots,x_n)\), which is a CDL by construction.
\end{definition}

When a poset \((P,\leq)\), and in particular a CDL, is finite, it can be represented by its \emph{Hasse diagram}. In such a diagram, we draw a point for each \(p \in P\) and an arrow from \(a\) to \(b\) whenever \(b\) \emph{covers} \(a\), that is, when \(a < b\) and there is no element \(c \in P\) such that \(a < c < b\). The elements are arranged so that the order increases in a fixed direction, typically upwards (or sometimes from left to right). For example, \Cref{fig:f2lattice} shows the Hasse diagram of \(\fdl(x,y)\).

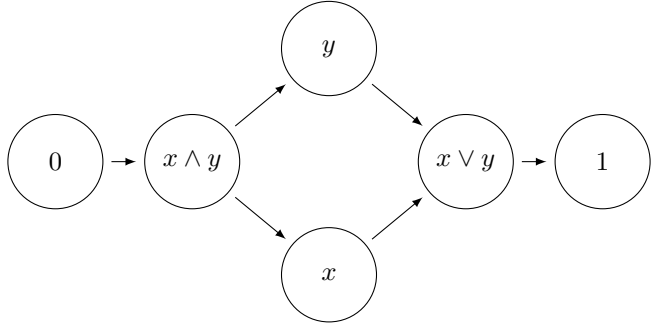
\begin{figure}
\centering
\begin{tikzpicture}[
  scale=1,
  every node/.style={
    circle,
    draw=black,
    fill=white,
    inner sep=0pt,
    minimum size=1.25cm,
    align=center
  },
  >=latex,
  shorten <=3pt, shorten >=3pt
]

\node (0) at (3.6,-7.5) {\(0\)};
\node (xy) at (5.4,-7.5) {\(x \wedge y\)};
\node (x) at (7.2,-9) {\(x\)};
\node (y) at (7.2,-6) {\(y\)};
\node (xvy) at (9,-7.5) {\(x \vee y\)};
\node (1) at (10.8,-7.5) {\(1\)};

\draw[->] (0) -- (xy);
\draw[->] (xy) -- (x);
\draw[->] (xy) -- (y);
\draw[->] (x) -- (xvy);
\draw[->] (y) -- (xvy);
\draw[->] (xvy) -- (1);
\end{tikzpicture}
\caption{Hasse diagram of \(\fdl(x,y)\), with increasing order from left to right.}
\label{fig:f2lattice}
\end{figure}

\begin{definition}[Meet-prime element]
Let \((L,\leq)\) be a CDL and \(\ell \in L\). We say that \(\ell\) is \emph{meet-prime} in \(L\) if \(a \wedge b \leq \ell\) implies \(a \leq \ell\) or \(b\leq \ell\). 
\end{definition}

\begin{definition}[Filter subsets]
\label{def:Lfilter}
Let \((P,\leq)\) be a poset and \(p \in P\). We define the filter subsets
\[
P^{\geq p} \! = \! \{a \in P \mid a \geq p\}, \;
P^{\leq p} \! = \! \{a \in P \mid a \leq p\}, \;
P^{> p} \! = \! \{a \in P \mid a > p\}, \;
P^{< p} \! = \! \{a \in P \mid a < p\}, \;
P^{= p} \! = \! \{a \in P \mid a = p\},
\]
and similarly the corresponding complements \(P^{\not\geq p}\), \(P^{\not\leq p}\), \(P^{\not> p}\), \(P^{\not< p}\), and \(P^{\neq p}\).
\end{definition} 

\section{Sets and L-fuzzy subsets}
\label{sec:lfuzzysubsets}

Let \(X\) be a non-empty set. The \emph{power set} of \(X\), denoted \(\calp(X)\), is the set of all subsets of \(X\).  
Each subset \(S \in \calp(X)\) can be identified with its characteristic function \(I_S : X \to \{0,1\}\), defined by \(I_S(x)=1\) if \(x \in S\) and \(I_S(x)=0\) otherwise. Replacing the Boolean lattice \((\{0,1\},\leq)\) with a general CDL \((L,\leq)\) leads to the following definition.

\begin{definition}[\(L\)-fuzzy subset]
\label{def:Lfuzzysubset}
Let \(X\) be a non-empty set and \((L, \leq)\) a CDL. An \emph{\(L\)-fuzzy subset} of \(X\) is a map \(\mu : X \to L\).  
The set of all \(L\)-fuzzy subsets of \(X\) is called the \emph{\(L\)-fuzzy power set} of \(X\), and is denoted \(\calfp(X, L)\).
\end{definition}

The classical subsets of \(X\) correspond to the particular case \(L=\{0,1\}\). In this situation, \(\calfp(X,\{0,1\})\) can be naturally identified with the power set \(\calp(X)\). For this reason, elements of \(\calp(X)\) are called \emph{crisp subsets} of \(X\), in contrast to the more general \(L\)-fuzzy subsets in \(\calfp(X,L)\).

For each \(x \in X\), the value \(\mu(x) \in L\) is called the \emph{membership value} of \(x\) with respect to \(\mu\). The interpretation is the following: if \(\mu(x)=0\), then \(x\) does not belong to \(\mu\); if \(\mu(x)=1\), then \(x\) fully belongs to \(\mu\); and if \(0 < \mu(x) < 1\), then \(x\) belongs to \(\mu\) to an intermediate degree. The image or
\emph{set of values of \(\mu\)} is \(L(\mu)=\{\mu(x)\mid x \in X\} \subseteq L\).

\begin{example}
\label{exa:zadehfuzzy}
Let \(([0,1], \leq)\) be the CDL of real numbers in the unit interval. The fuzzy sets introduced by Zadeh in~\cite{zadeh1965fuzzy} are precisely the elements of \(\calfp(X, [0,1])\).
\end{example}

We now extend the basic notions in set theory to the \(L\)-fuzzy setting.

\begin{definition}[Operations on \(L\)-fuzzy subsets]
\label{def:fuzzy_operations}
Let \(\mu, \nu \in \calfp(X,L)\), \(\psi \in \calfp(Y,L)\) and \(f:X \to Y\). Then,
\begin{itemize}
\item (Inclusion) \(\mu \subseteq \nu\) if and only if \(\mu(x) \leq \nu(x)\) for all \(x \in X\).
\item (Union) \(\mu \cup \nu \in \calfp(X,L)\) is defined as \((\mu \cup \nu)(x) = \mu(x) \vee \nu(x)\) for any \(x \in X\).
\item (Intersection) \(\mu \cap \nu \in \calfp(X,L)\) is defined as \((\mu \cap \nu)(x) = \mu(x) \wedge \nu(x)\) for all \(x \in X\).
\item (Cartesian product) \(\mu \times \psi \in \calfp(X \times Y,L)\) is defined as \((\mu \times \psi)(x,y) = \mu(x) \wedge \psi(y)\) for all \((x,y) \in X \times Y\).
\item (Image) \(f(\mu) \in \calfp(Y,L)\) is defined as \(f(\mu)(y) = \bigvee \{\mu(x) \mid f(x) = y\}\) for all \(y \in Y\).
\item (Preimage) \(f^{-1}(\psi) \in \calfp(X,L)\) is defined as \(f^{-1}(\psi)(x) = \psi(f(x))\) for all \(x \in X\).
\end{itemize}
\end{definition}

We now present some \(L\)-fuzzy subsets of interest for this paper.

\begin{definition}
\label{def:fuzzysingleton}
Let \(S \subseteq X\) and a non-zero value \(\ell \in L\). The \(L\)-fuzzy subset \(S_{\ell} \in \calfp(X, L)\) is defined by \(S_{\ell}(x) = \ell\) if \(x \in S\) and \(S_{\ell}(x)=0\) otherwise. 
If \(S\) is a singleton \(\{a\}\), we say that \(\{a\}_{\ell}\) is an \emph{\(L\)-fuzzy singleton}.
\end{definition}

\begin{example}
\label{exa:fdlfuzzy}
A \emph{chromatic dataset} is a pair \(\mathcal{D} = (X,f)\), where  
\(X = \{x_i \in \bbr^d \mid i = 1,\dots,n\}\) is a finite dataset and  
\(f : X \to C = \{c_1,\dots,c_k\}\) is a labeling function assigning a class label or ``color'' to each data instance. Let \(\fdl(c_1,\ldots,c_k) = (L,\leq)\) be the free distributive lattice generated by \(C\). Since \(C \subseteq L\), the labeling function can be regarded as a map \(f : X \to L\), and therefore the chromatic dataset \(\mathcal{D} = (X,f)\) can be interpreted as an \(L\)-fuzzy subset \(f \in \calfp(X,L)\), where each data instance is assigned the lattice element corresponding to its color.
\end{example}

In \Cref{def:Lfilter} we saw how to filter the elements of a lattice according to their relation to a certain element. Now, we extend the notion of filtering to \(L\)-fuzzy subsets.

\begin{definition}[Filter subsets in an \(L\)-fuzzy subset]
\label{def:mufilter}
Let \(\mu \in \calfp(X,L)\) and let \(\ell \in L\). We define the following crisp subsets of \(X\), known as the filter subsets of \(\mu\), by
\[
\mu^{\geq \ell} = \mu^{-1}(L^{\geq \ell}) = \{x \in X \mid \mu(x) \geq \ell\}, \quad
\mu^{\leq \ell} = \mu^{-1}(L^{\leq \ell}), \quad
\mu^{> \ell} = \mu^{-1}(L^{> \ell}), \quad
\mu^{< \ell} = \mu^{-1}(L^{< \ell}), \quad
\mu^{= \ell} = \mu^{-1}(L^{= \ell}),
\]
and similarly for the complementary subsets 
\(\mu^{\not\geq \ell}\), \(\mu^{\not\leq \ell}\), \(\mu^{\not> \ell}\), \(\mu^{\not< \ell}\), and \(\mu^{\neq \ell}\). Two notable filter subsets of \(X\) are the \emph{support of \(\mu\)}, defined as \(\mu^* = \mu^{>0}=\mu^{\neq 0}\), and the \emph{core of \(\mu\)}, defined as \(\mu_* = \mu^{=1} = \mu^{\geq 1}\).
\end{definition}

Among these subsets, the filters \(\mu^{\geq \ell}\), also known as \emph{cuts}, play a central role. In the classical case of fuzzy sets valued in \(([0,1],\leq)\), the set \(\mu^{\geq \alpha}\) for \(\alpha \in [0,1]\) is known as the \(\alpha\)-cut of \(\mu\). More generally, the family \(\{\mu^{\geq \ell} \mid \ell \in L\} \subset \calp(X)\) forms a system of crisp subsets of \(X\) that encodes the entire structure of \(\mu\). Indeed, the membership function \(\mu\) can be recovered from it by \(\mu(x) = \bigvee \{\ell \in L \mid x \in \mu^{\geq \ell}\}\). Thus, instead of studying the \(L\)-fuzzy subset \(\mu\) directly, one may equivalently study the family of its cuts. This perspective allows us to provide a bridge between \(L\)-fuzzy and crisp objects. We now present some basic properties of cuts.

\begin{proposition}
\label{prop:subsetsproperties}
Let \(X\) be a non-empty set and \((L,\leq)\) a CDL.
\begin{enumerate}
\item If \(\mu \in \calfp(X,L)\), for any \(\ell_1,\ell_2 \in L\) such that \(\ell_1 \leq \ell_2\), we have \(\mu^{\geq \ell_1} \supseteq \mu^{\geq \ell_2}\).
\item If \(\mu \in \calfp(X,L)\), for any subset \(S \subseteq L\) we have \(\mu^{\geq \bigvee S} = \bigcap_{\ell \in S}\mu^{\geq \ell}\).
\item If \(\mu \in \calfp(X,L)\), for any subset \(S \subseteq L\) we have \(\mu^{\geq \bigwedge S} \supseteq \bigcup_{\ell \in S}\mu^{\geq \ell}\).
\end{enumerate}   
\end{proposition}

\begin{proof}
We prove each statement in turn.
\begin{enumerate}
\item Let \(\ell_1 \leq \ell_2\) and suppose \(x \in \mu^{\geq \ell_2}\). Then \(\mu(x) \geq \ell_2 \geq \ell_1\), so \(x \in \mu^{\geq \ell_1}\). Therefore \(\mu^{\geq \ell_1} \supseteq \mu^{\geq \ell_2}\).
\item We prove both inclusions. If \(x \in \mu^{\geq \bigvee S}\), then \(\mu(x)\geq \bigvee S\), which implies \(\mu(x) \geq \ell\) for all \(\ell \in S\). Thus \(x \in \bigcap_{\ell \in S}\mu^{\geq \ell}\). Conversely, if \(x \in \bigcap_{\ell \in S}\mu^{\geq \ell}\), then \(\mu(x) \geq \ell\) for all \(\ell \in S\). By definition of the join, this implies \(\mu(x) \geq \bigvee S\), so \(x \in \mu^{\geq \bigvee S}\). Therefore \(\mu^{\geq \bigvee S} = \bigcap_{\ell \in S}\mu^{\geq \ell}\).
\item Let \(x \in \bigcup_{\ell \in S}\mu^{\geq \ell}\). There exists \(\ell \in S\) such that \(x \in \mu^{\geq \ell}\). Then, \(\mu(x) \geq \ell\geq \bigwedge S\) and hence \(x \in \mu^{\geq \bigwedge S}\). Thus \(\mu^{\geq \bigwedge S} \supseteq \bigcup_{\ell \in S}\mu^{\geq \ell}\). \qedhere
\end{enumerate}
\end{proof}

The first item of \Cref{prop:subsetsproperties} shows that there exists a contravariant functor \(\cut(\mu) : L \to \calp(X)\), given by \(\ell \mapsto \mu^{\ge \ell}\), from the CDL \((L,\le)\) (viewed as a category with one arrow \(\ell_1 \to \ell_2\) whenever \(\ell_1 \le \ell_2\)) to the category \((\mathcal{P}(X),\subseteq)\) of subsets of \(X\) (with morphisms restricted to inclusions). 
The second item shows that \(\cut(\mu)\) sends joins in the lattice \(L\) to meets (intersections) in \(\mathcal{P}(X)\). However, the dual statement need not hold in general; meets in \(L\) do not necessarily correspond to joins (unions) in \(\mathcal{P}(X)\).

\section{Modules and L-fuzzy submodules}
\label{sec:lfuzzysubmodules}

The definition of simplicial homology requires the notion of a module over a ring. 
Throughout this paper, we assume that \((\bba,+,\cdot)\) is a commutative ring with neutral element \(0 \in \bba\) for addition and neutral element \(1 \in \bba\) for multiplication (not to be confused with the elements \(0,1 \in L\)). A class of rings that plays a central role in our work is that of \emph{principal ideal domains} (PIDs) \(\bbd\), in which every ideal \(I \subseteq \bbd\) is principal, that is, there exists \(a \in \bbd\) such that \(I = (a)\). 
Examples of PIDs include the integer ring \(\bbz\) and the quotient rings \(\bbz/(n)\) for \(n \in \bbz\). A particularly important subclass of PIDs is that of \emph{fields} \(\bbf\), in which every non-zero element \(a \in \bbf\) is a unit, that is, there exists \(a^{-1} \in \bbf\) such that \(aa^{-1} = a^{-1}a = 1\). Typical examples of fields include the rational numbers \(\mathbb{Q}\), the real numbers \(\bbr\), and the finite fields \(\bbz/(p)\) for \(p\) prime.

\begin{definition}[Module]
\label{def:module}
Let \(\bba\) be a ring. An \(\bba\)-module is a non-empty set \(\Mo\) equipped with an addition \(+\colon \Mo \times \Mo \to \Mo\) and a scalar multiplication \(\cdot\colon \bba \times \Mo \to \Mo\) such that:
\begin{enumerate}
\item \((\Mo,+)\) is an abelian group with neutral element \(0 \in \Mo\).
\item \(a \cdot (m_1+m_2) = a \cdot m_1 + a \cdot m_2\) for all \(m_1,m_2 \in \Mo\) and all \(a \in \bba\).
\item \((a_1+a_2) \cdot m = a_1 \cdot m + a_2 \cdot m\) for all \(m \in \Mo\) and all \(a_1,a_2 \in \bba\).
\item \((a_1\cdot a_2) \cdot m = a_1 \cdot (a_2 \cdot m)\) for all \(m \in \Mo\) and all \(a_1,a_2 \in \bba\).
\item \(1 \cdot m = m\) for all \(m \in \Mo\).
\end{enumerate}
\end{definition}

The product in \(\bba\) and the scalar product in \(\Mo\) can also be written by juxtaposition. Then, \(a_1 \cdot a_2 \in \bba\) can be written as \(a_1a_2\) and \(a\cdot m \in \Mo\) can be written as \(am\).
When the chosen ring is a field \(\bbf\), then \(\Mo\) is an \(\bbf\)-vector space.

\begin{definition}[Crisp submodule]
\label{def:submodule}
Let \(\Mo\) be an \(\bba\)-module. A subset \(\No \subseteq \Mo\) is called a \emph{crisp submodule} of \(\Mo\) if \(0 \in \No\), \(m_1+m_2 \in \No\) for all \(m_1,m_2 \in \No\), and \(am \in \No\) for all \(a \in \bba\) and all \(m \in \No\). The set of all crisp submodules of \(\Mo\) is denoted by \(\sub(\Mo)\).
\end{definition}

In other words, a crisp submodule is a crisp subset \(\No \subseteq \Mo\) that is an \(\bba\)-module itself with the addition and scalar product induced from \(\Mo\). 
The following notion of \(L\)-fuzzy submodule is taken from \cite{mashinchi1992fuzzy}.

\begin{definition}[\(L\)-fuzzy submodule]
\label{def:fuzzy_submodule}
Let \(\Mo\) be an \(\bba\)-module and \((L,\leq)\) a CDL. An \(L\)-fuzzy subset \(\mu \in \calfp(\Mo,L)\) is called an \emph{\(L\)-fuzzy submodule} of \(\Mo\) if it satisfies \(\mu(0)=1\), \(\mu(m_1+m_2) \geq \mu(m_1) \wedge \mu(m_2)\) for all \(m_1,m_2 \in \Mo\) and \(\mu(a \cdot m) \geq \mu(m)\) for all \(m \in \Mo\) and all \(a \in \bba\). The set of all \(L\)-fuzzy submodules of \(\Mo\) is denoted by \(\calfm(\Mo,L)\).
\end{definition}

If the ring is a field \(\bbf\), then \(\Mo\) is an \(\bbf\)-vector space and any \(\mu \in \calfm(\Mo,L)\) is called an \emph{\(L\)-fuzzy subspace of $\Mo$}. In this case, since every non-zero scalar \(a \in \bbf\) is a unit, for any \(m \in \Mo\), we have
\( \mu(m) \leq \mu(a \cdot m) \leq \mu(a^{-1} \cdot a \cdot m) = \mu(m) \)
and therefore \(\mu(a \cdot m) = \mu(m)\). Consequently, in the case of vector spaces, the third condition in \Cref{def:fuzzy_submodule} is equivalent to
\(\mu(a \cdot m) = \mu(m) \text{ for all } m \in \Mo \text{ and all } a \in \bbf \setminus \{0\}\).
We now state some basic properties of \(L\)-fuzzy submodules.

\begin{proposition}
\label{prop:submodules_properties}
Let \(\Mo\) be an \(\bba\)-module and \((L,\leq)\) a CDL.
\begin{enumerate}
\item Let \(\{\mu_i \mid i \in I\} \subseteq \calfm(\Mo,L)\). Then, \(\mu_{\wedge} = \bigcap_{i\in I} \mu_i \in \calfm(\Mo,L)\).
\item For any \(\mu \in \calfm(\Mo,L)\) and any \(\ell \in L\), the upper level set \(\mu^{\geq \ell} \subseteq \Mo\) is a crisp submodule of \(\Mo\).
\end{enumerate}
\end{proposition}

\begin{proof} We prove each statement in turn.
\begin{enumerate}
\item In the first place, we have \(\mu_{i}(0)=1\) for all \(i \in I\), which implies \(\mu_{\wedge}(0)=1\). Consider now \(m_1,m_2 \in \Mo\). We have \(\mu_i (m_1+m_2) \geq \mu_i(m_1) \wedge \mu_i(m_2)\) for all \(i \in I\). Therefore,
\[
\mu_{\wedge}(m_1+m_2) = \bigwedge_{i\in I} \mu_i (m_1+m_2) \geq \bigwedge_{i \in I} \left(\mu_i(m_1) \wedge \mu_i(m_2) \right) = \left(\bigwedge_{i \in I} \mu_i(m_1) \right) \wedge \left(\bigwedge_{i \in I} \mu_i(m_2) \right) = \mu_{\wedge}(m_1) \wedge \mu_{\wedge}(m_2).
\]
Consider now any \(m \in \Mo\) and any \(a \in \bba\). We have \(\mu_i (a \cdot m) \geq \mu_i(m)\) for all \(i \in I\). Therefore, \(\mu_{\wedge}(a \cdot m) = \bigwedge_{i\in I} \mu_i (a \cdot m) \geq \bigwedge_{i \in I} \mu_i(m) = \mu_{\wedge}(m)\). Thus, \(\mu_{\wedge} \in \calfm(\Mo,L)\).
\item In the first place, \(0 \in \mu^{\geq \ell}\) because \(\mu(0)=1 \geq \ell\). Consider now \(m_1,m_2 \in \mu^{\geq \ell}\). This implies that \(\mu(m_1) \geq \ell\) and \(\mu(m_2)\geq \ell\). Therefore, \(\mu(m_1+m_2) \geq \mu(m_1) \wedge \mu(m_2) \geq \ell\) and \(m_1+m_2 \in \mu^{\geq \ell}\). Finally, consider any \(m \in \mu^{\geq \ell}\) and any scalar \(a \in \bba\). We have \(\mu(a \cdot m) \geq \mu(m) \geq \ell\), so \(a \cdot m \in \mu^{\geq \ell}\). Thus, \(\mu^{\geq \ell}\) is a crisp submodule of \(\Mo\). \qedhere
\end{enumerate}
\end{proof}

The second item shows that the contravariant functor 
\(\cut(\mu): L \to \calp(\Mo)\) defined by \(\mu \in \calfm(\Mo,L)\) factors through the category \((\sub(\Mo),\subseteq)\), where morphisms are given by inclusions. That is, there exists a functor \(g: L \to \sub(\Mo)\) such that \(\cut(\mu)= i \circ g\), where \(i: \sub(\Mo) \to \calp(\Mo)\) denotes the inclusion functor. Accordingly, \(\cut(\mu)\) can be identified with \(g\), and we may equivalently view it as a functor \(\cut(\mu): L \to \sub(\Mo)\).

\begin{remark}
\label{rem:supportcorelevelset}
It follows from \Cref{prop:subsetsproperties} that the core \(\mu_*\) is always a crisp submodule of \(\Mo\) because \(\mu_* = \mu^{\geq 1}\). However, the support \(\mu^* = \mu^{>0}\) is not necessarily a submodule of \(\Mo\). For instance, consider the CDL \((L,\leq)\), where \(L = \{0, x, y, 1\}\) and \(0 \leq x, y \leq 1\), but \(x\) and \(y\) are incomparable. In this CDL, we have \(x \wedge y = 0\) and \(x \vee y = 1\). Now, consider the following \(L\)-fuzzy subspace \(\mu \in \calfm(\bbr^2,L)\):
\[
\mu((r_1,r_2)) =
\begin{cases}
1 & \text{if } r_1 = r_2 = 0,\\
x & \text{if } r_2 = 0 \text{ and } r_1 \neq 0,\\
y & \text{if } r_1 = 0 \text{ and } r_2 \neq 0,\\
0 & \text{otherwise.}
\end{cases}
\]
Its support is the set \(\mu^* = \{(r_1, r_2) \in \bbr^2 \mid r_1 \cdot r_2 = 0\}\), which is not a crisp subspace of \(\bbr^2\). 
\end{remark}

The following proposition gives a sufficient condition for \(\mu^*\) to be a crisp submodule.

\begin{proposition}
\label{prop:meet-prime}
Let \(\mu \in \calfm(\Mo,L)\). If \(0\) is a meet-prime element in \(L\), the support \(\mu^*\) is a crisp submodule of \(\Mo\).
\end{proposition}

\begin{proof}
In the first place, \(0 \in \mu^*\) because \(\mu(0)=1 >0\). Consider now \(m_1,m_2 \in \mu^*\). This implies that \(\mu(m_1)>0\) and \(\mu(m_2)>0\). Since \(0\) is meet-prime in \(L\), the meet of two non-zero elements is also non-zero. Therefore, \(\mu(m_1+m_2) \geq \mu(m_1) \wedge \mu(m_2)>0\) and \(m_1+m_2 \in \mu^*\). Finally, consider any \(m \in \mu^*\) and any scalar \(a \in \bba\). We have \(\mu(am) \geq \mu(m)>0\), so \(am \in \mu^*\). Thus, \(\mu^*\) is a crisp submodule of \(\Mo\).
\end{proof}

The element \(0\) is meet-prime in some CDLs, such as \(([0,1],\leq)\) and \(\fdl(x_1,\dots,x_n)\), but it is not meet-prime in the CDL described in \Cref{rem:supportcorelevelset}. Therefore, in order to apply concepts whose definitions depend on the support, we replace the subset \(\mu^*\) with an appropriate submodule generated by it.

\begin{definition}[Generated crisp submodule]
\label{def:generated}
Let \(\Mo\) be an \(\bba\)-module and let \(S \subseteq \Mo\).
The \emph{crisp submodule generated by \(S\)}, denoted \(\langle S \rangle\), is defined as \(\langle S \rangle = \bigcap \{\No \in \sub(\Mo) \mid S \subseteq \No\}\).
\end{definition}

As a consequence of \Cref{prop:submodules_properties}, \(\langle S \rangle\) is indeed a submodule because it is the intersection of an arbitrary collection of submodules, and it is by definition the smallest submodule of \(\Mo\) that contains \(S\). Replacing the crisp subset \(S\) by an \(L\)-fuzzy subset \(\mu\) leads to the following definition.

\begin{definition} [Generated \(L\)-fuzzy submodule \cite{ashok2020fuzzy}]
\label{def:generatedsubmodule}
Let \(\Mo\) be an \(\bba\)-module, \((L,\leq)\) a CDL and let \(\mu \in \calfp(\Mo,L)\). The \emph{\(L\)-fuzzy submodule generated by \(\mu\)}, denoted \(\langle \mu \rangle\), is defined as \(\langle \mu \rangle = \bigcap \left\{ \nu \in \calfm(\Mo,L) \mid \mu \subseteq \nu \right\}\).
\end{definition}

Again, \(\langle \mu \rangle\) is the smallest \(L\)-fuzzy submodule containing \(\mu\). We now provide a more constructive definition of \(\langle \mu \rangle\). This is a particular case of \cite[Theorem~3.4]{addis2021commutator}, stated there in the general setting of universal algebras, which we rewrite here in our notation.

\begin{theorem}[\cite{addis2021commutator}]
\label{thm:generatedconstructive}
Let \( \mu \in \calfp(\Mo,L) \) and \(m \in \Mo\). Then, we have
\[
\langle \mu \rangle (m)
= \bigvee \Big\{\,\bigwedge_{i=1}^{n} \mu(m_i)\;\Big|\;
m = \sum_{i=1}^{n} a_i m_i,\ a_i \in \bba,\ m_i \in \Mo,\ n \ge 0 \Big\},
\]
where the case \(n = 0\) corresponds to the empty sum \(\sum \emptyset = 0\), considered only for \(m = 0\).
\end{theorem}

That is, for any \(m \in \Mo\), we consider all possible finite linear combinations \(m = \sum_{i=1}^{n} a_i m_i,\)
where \(a_i \in \bba\), \(m_i \in \Mo\), and \(n \geq 0\). Each linear combination has an associated value \(\bigwedge_{i=1}^{n} \mu(m_i) \in L\), and \(\langle \mu \rangle(m)\) is defined as the join of all these values.  
The following result applies \Cref{thm:generatedconstructive} to give a explicit description of \(\langle \mu \rangle\) when \(\mu \in \calfp(\Mo,L)\) is a finite union of \(L\)-fuzzy singletons.

\begin{corollary}
\label{cor:singletons_generating}
Let \(\Mo\) be an \(\bba\)-module, \((L,\leq)\) a CDL, and let 
\(E = \{e_1, \ldots, e_k\} \subset \Mo\) be a linearly independent set.  
Given \(\ell_1, \ldots, \ell_k \in L\), define the union of \(L\)-fuzzy singletons 
\(\mu = \bigcup_{i=1}^k \{e_i\}_{\ell_i}\). If \(m \in \langle E \rangle\), let 
\(m = \sum_{i=1}^k a_i e_i\) with \(a_i \in \bba\) be the unique linear combination of elements of \(E\) representing \(m\). 
Then
\[
\langle \mu \rangle(m) =
\begin{cases}
\displaystyle \bigwedge_{\substack{i=1,\ldots,k \\ a_i \neq 0}} \ell_i 
& \text{if } m \in \langle E \rangle, \\[0.8em]
0 
& \text{if } m \notin \langle E \rangle.
\end{cases}
\]
\end{corollary}

\begin{proof}
Let \(m \in \Mo\), and consider any finite linear combination 
\(m = \sum_{i=1}^{n} a_i m_i\) with \(m_i \in \Mo\) and \(a_i \in \bba\).  
If there exists some \(m_j \notin E\), 
by \Cref{def:fuzzysingleton}, it holds that \(\{e_i\}_{\ell_i}(m_j)=0\) for all \(i=1,\dots,k\). Consequently, 
\(\mu(m_j) = \bigvee_{i=1}^k 0=0\) and the value associated to the linear combination is \(\bigwedge_{i=1}^{n} \mu(m_i) = 0\). It remains to show that the formula for \(\langle \mu\rangle\) holds for the two possible cases:
\begin{enumerate}[label=\roman*)]
\item If \(m \notin \langle E \rangle\), it cannot be expressed as a linear combination of elements of \(E\).  
Hence, every finite linear combination representing \(m\) includes at least one element outside \(E\), and therefore \(\langle \mu \rangle(m) = \bigvee \{0\} = 0\).
\item If \(m \in \langle E \rangle\), then \(m\) has a unique representation 
\(m = \sum_{i=1}^k a_i e_i\) with \(a_i \in \bba\). 
We can also write the reduced linear combination 
\(m=\sum_{i=1,\ldots,k}^{a_i \neq 0} a_i e_i\),
whose associated value is
\(\bigwedge_{i=1,\ldots,k}^{a_i \neq 0} \ell_i\).  
Any other finite linear combination with sum \(m\) either contains elements outside \(E\), yielding value \(0\), or contains more elements of \(E\) than strictly needed, yielding a smaller meet.  
Thus, the join of all possible values is precisely
\(\langle \mu \rangle(m) = \bigwedge_{i=1,\ldots,k}^{a_i \neq 0} \ell_i\). \qedhere
\end{enumerate}
\end{proof}

We now recall the concept of module homomorphism to study how does it interact with \(L\)-fuzzy submodules.

\begin{definition}[Homomorphism of crisp modules]
\label{def:homomorphism}
Let \(\Mo, \No\) be two \(\bba\)-modules. The map \(f:\Mo \to \No\) is called a \emph{homomorphism of modules} if \(f(0)=0\), \(f(m_1+m_2)=f(m_1)+f(m_2)\) for all \(m_1,m_2 \in \Mo\) and \(f(am)=af(m)\) for all \(m \in \Mo\) and all \(a \in \bba\). A bijective homomorphism is called an \emph{isomorphism}. If such an isomorphism \(f: \Mo \to \No\) exists, we say that \(\Mo\) and \(\No\) are \emph{isomorphic} and write \(\Mo \cong \No\).
\end{definition}

This definition is extended to \(L\)-fuzzy submodules by adding a compatibility condition on their respective maps.

\begin{definition}[Homomorphism of \(L\)-fuzzy submodules \cite{pan1987fuzzy}]
\label{def:lfuzzyhomomorphism}
Let \(\Mo, \No\) be two \(\bba\)-modules, and let \(\mu \in \calfm(\Mo,L)\), \(\nu \in \calfm(\No,L)\). 
A homomorphism of \(\bba\)-modules \(f:\Mo \to \No\) is called a \emph{homomorphism of \(L\)-fuzzy submodules} if \(\mu (m) \leq \nu(f(m))\) for all \(m \in \Mo\). 
An \emph{isomorphism of \(L\)-fuzzy submodules} is an isomorphism \(f:\Mo \to \No\) with \(\mu (m) = \nu(f(m))\) for all \(m \in \Mo\). If such an isomorphism exists, we say that \(\mu\) and \(\nu\) are \emph{isomorphic} and write \(\mu \cong \nu\).
\end{definition}

The following results present basic properties of the interaction between homomorphisms, isomorphisms, and \(L\)-fuzzy submodules, with particular emphasis on images, preimages, and cuts.

\begin{proposition}(\cite{zahedi1993some})
\label{prop:extension_submodules}
Let \(f : \Mo \to \No\) be a homomorphism of \(\bba\)-modules. If \(\mu \in \calfm(\Mo, L)\) and \(\nu \in \calfm(\No, L)\), then \(f(\mu) \in \calfm(\No, L)\) and \(f^{-1}(\nu) \in \calfm(\Mo, L)\).
\end{proposition}

\begin{proposition}
\label{prop:level_iso}
Let \(\mu \in \calfm(\Mo,L)\) and \(\nu \in \calfm(\No,L)\). If the map \(f: \Mo \to \No\) is an isomorphism of \(L\)-fuzzy submodules, then \(\mu^{\geq \ell}\cong\nu^{\geq \ell}\) for every \(\ell \in L\).
\end{proposition}

\begin{proof}
For any \(m\in\Mo\), we have
\[
m\in\mu^{\ge\ell}
\;\Longleftrightarrow\;
\mu(m)\ge\ell
\;\Longleftrightarrow\;
\nu(f(m))\ge\ell
\;\Longleftrightarrow\;
f(m)\in\nu^{\ge\ell}.
\]
Hence \(f(\mu^{\ge\ell})=\nu^{\ge\ell}\). Since \(f\) is bijective, the restriction \(f|_{\mu^{\ge\ell}} :\mu^{\ge\ell}\to\nu^{\ge\ell}\) is bijective too, and therefore \(\mu^{\ge\ell}\cong\nu^{\ge\ell}\).
\end{proof}

We continue with the definition of quotient, which is crucial to define \(L\)-fuzzy simplicial homology in \Cref{sec:fuzhom}.

\begin{definition}[Quotient of crisp submodules]
\label{def:quotient}
Let \(\Mo\) be an \(\bba\)-module and let \(\No \in \sub(\Mo)\). The \emph{quotient} \(\Mo/\No\) is the \(\bba\)-module whose elements are the cosets \([m] = m + \No=\{m+n\mid n \in \No\}\), with addition defined by \([m_1]+[m_2] = [m_1+m_2]\) and scalar product defined by \(a[m]=[am]\). 
\end{definition}

Given an element \(m \in \Mo\), the coset \([m] \in \Mo/\No\) is called the \emph{class of \(m\)}. The classes \([m_1]\) and \([m_2]\) are equal if and only if \(m_1 - m_2 \in N\).
We now introduce the quotient of \(L\)-fuzzy submodules. 
The original source \cite{ashok2020fuzzy} defines it in terms of the support \(\mu^*\) because \(0\) is meet-prime in the CDL \(([0,1],\leq)\) and therefore \(\mu^*\) is indeed a submodule. To obtain a definition that is valid for an arbitrary CDL, we replace \(\mu^*\) with \(\langle \mu^* \rangle\).

\begin{definition}[Quotient of \(L\)-fuzzy submodules]
\label{def:lfuzzyquotient}
Let \(\Mo\) be an \(\bba\)-module, \((L,\leq)\) a CDL and \(\mu, \nu \in \calfm(\Mo,L)\) two \(L\)-fuzzy submodules such that \(\mu \subseteq \nu\). The \emph{quotient of \(\nu\) with respect to \(\mu\)} is the \(L\)-fuzzy submodule \(\nu/\mu \in \calfm(\langle\nu^*\rangle/\langle\mu^*\rangle,L)\) such that for each \(m \in \langle\nu^*\rangle\):
\[
(\nu/\mu)([m]) = \bigvee \{\nu(n) \mid n \in [m]\}.
\]
\end{definition}

\begin{remark}
Note that this definition does not use the values of \(\mu\), only its support \(\mu^*\). That is, if we had two different \(L\)-fuzzy submodules \(\mu_1,\mu_2 \subseteq \nu\) with \(\langle\mu_1^*\rangle = \langle\mu_2^*\rangle\), then \(\nu/\mu_1 = \nu/\mu_2\). It may be interesting for future work to develop a definition of quotient between two \(L\)-fuzzy submodules that actually uses the values of the denominator.
\end{remark}

We conclude this section by giving some results to classify modules and \(L\)-fuzzy submodules. We say that an \(\bba\)-module \(\Mo\) is \emph{finitely generated} if there exists a finite subset \(S \subseteq \Mo\) such that \(\Mo = \langle S \rangle\). The following theorem provides a complete classification of finitely generated \(\bbd\)-modules (being \(\bbd\) a PID) up to isomorphism.

\begin{theorem}[Structure theorem for finitely generated \(\bbd\)-modules \cite{dummit2004chapter12}] 
\label{thm:modules_structure}
Let \(\bbd\) be a PID and \(\Mo\) a finitely generated \(\bbd\)-module. Then, there exists 
a unique \(\beta \in \bbz_{\geq0}\) and 
a unique sequence $a_1,\ldots,a_m$ of non-zero and non-unit elements of $\bbd$ with \(a_1 \mid a_2 \mid \cdots \mid a_m\) such that:
\[
\Mo \cong \bbd^{\beta} \oplus \bigoplus_{j=1}^m\bbd/(a_j).
\]
\end{theorem}

By this theorem, every finitely generated \(\bbd\)-module \(\Mo\) is completely determined (up to isomorphism) by its \emph{Betti number} \(\beta\), which defines the free submodule \(\bbd^{\beta}\), and its \emph{torsion coefficients} \(a_1,\ldots,a_m\), which describe the torsion submodule \(\bigoplus_{j=1}^m \bbd/(a_j)\).
The Betti number \(\beta\) is called the \emph{rank} of \(\Mo\), and we write \(\rank(\Mo) = \beta\).
When the chosen PID is a field \(\bbf\), the module \(\Mo\) is in fact an \(\bbf\)-vector space. Since every non-zero element of \(\bbf\) is a unit, the torsion submodule of \(\Mo\) is trivial and \(\Mo\) is completely determined by its rank.

To the best of our knowledge, there is no structure theorem for \(L\)-fuzzy submodules analogous to that for crisp modules, but we can still use existing results to distinguish them. By \Cref{prop:level_iso}, isomorphic \(L\)-fuzzy submodules define isomorphic cuts at every \(\ell \in L\). In particular, let \(\mu \in \calfm(\Mo,L)\) and \(\nu \in \calfm(\No,L)\) be two \(L\)-fuzzy submodules. If there exists some \(\ell \in L\) such that \(\mu^{\ge \ell}\) and \(\nu^{\ge \ell}\) have different Betti numbers or different torsion coefficients, then \Cref{thm:modules_structure} implies that \(\mu^{\ge \ell} \not\cong \nu^{\ge \ell}\), and \Cref{prop:level_iso} then guarantees that \(\mu \not\cong \nu\). 

Focusing only on the Betti numbers, any \(\mu \in \calfm(\Mo,L)\) defines a contravariant functor \(\rank\!\cut(\mu): L \to \bbz_{\ge 0}\), given by \(\ell \mapsto \rank(\mu^{\ge \ell})\), from the CDL \((L,\leq)\) to the set of non-negative integers \((\bbz_{\ge 0},\le)\), both viewed as categories with arrows defined by the order relation. Our discussion shows that the contravariant functor \(\rank\!\cut(\mu)\) is invariant under isomorphisms and therefore provides a practical tool to distinguish non-isomorphic \(L\)-fuzzy submodules.

\section{Simplicial complexes and L-fuzzy subcomplexes}
\label{sec:lfuzzysubcomplexes}

To define simplicial homology, we first introduce the geometric structures on which it is built, namely simplicial complexes, which model topological spaces as a collection of elementary pieces arranged in a controlled way. In this text, we always assume that simplicial complexes are finite.

\begin{definition}[Simplicial complex~\cite{croom2012basic}]
\label{def:simplicial_complex}
Let \(n\in \bbz_{\geq1}\).
\begin{itemize}
\item Given a set \(S=\{v_0,\dots,v_d\}\) of \(d+1\) affinely independent points in \(\bbr^n\), the \emph{\(d\)-simplex} \(\sigma = \langle v_0,\dots,v_d\rangle\) is the convex hull of \(S\). 
The set \(S\) is called the \emph{vertex set of \(\sigma\)} and its elements are called \emph{vertices}. 1-simplices are tipically called \emph{edges} and 2-simplices are tipically called \emph{triangles}.
\item Let \(\sigma_1\) and \(\sigma_2\) be two simplices in \(\bbr^n\). 
We say that \(\sigma_1\) is a \emph{face of} \(\sigma_2\) if the vertex set of \(\sigma_1\) is contained in the vertex set of \(\sigma_2\). If the inclusion is strict, \(\sigma_1\) is a \emph{proper face of} \(\sigma_2\).
\item A \emph{simplicial complex} is a finite family \(\Delta\) of simplicial complexes in \(\bbr^n\) such that all the faces of a simplex in \(\Delta\) also belong to \(\Delta\) and the intersection of any two simplices in \(\Delta\) is either empty or a common face. The union of the vertex sets of all the simplices in \(\Delta\) is called the \emph{vertex set of \(\Delta\)}.
\item The set of \(d\)-simplices of \(\Delta\) is denoted \(\Delta_d\). The \emph{dimension} of \(\Delta\) is the maximum \(d \geq 0\) such that \(\Delta_d \neq \emptyset\).
\end{itemize}
\end{definition}

Simplicial complexes are particularly useful for studying topological spaces from a combinatorial viewpoint. 
If a topological space is homeomorphic to the union of the simplices of a simplicial complex, then many of its topological properties can be analyzed by working with the complex.

\begin{definition}[Crisp subcomplex]
\label{def:subcomplex}
Let \(\Delta\) be a simplicial complex. A non-empty subset \(\Gamma \subseteq \Delta\) is called a \emph{crisp subcomplex} of \(\Delta\) if all the faces of a simplex in \(\Gamma\) also belong to \(\Gamma\) and the intersection of any two simplices in \(\Gamma\) is either empty or a common face. The set of all crisp subcomplexes of \(\Delta\) is denoted by \(\sub(\Delta)\).
\end{definition}

In other words, a crisp subcomplex is a crisp subset \(\Gamma \subseteq \Delta\) that is a simplicial complex itself. The following notion of \(L\)-fuzzy subcomplex is inspired by the fuzzy simplicial sets introduced by Spivak \cite{spivak2009metric}, replacing \(([0,1],\leq)\) by a general CDL.

\begin{definition}[\(L\)-fuzzy subcomplex]
\label{def:lfuzzysubcomplex}
Let \(\Delta\) be a simplicial complex and let \((L,\leq)\) be a CDL. An \(L\)-fuzzy subset \(\mu \in \calfp(\Delta,L)\) is called an \emph{\(L\)-fuzzy subcomplex} of \(\Delta\) if, for every pair of simplices \(\sigma_1, \sigma_2 \in \Delta\) with \(\sigma_1 \subseteq \sigma_2\), we have \(\mu(\sigma_1) \geq \mu(\sigma_2)\). The set of all \(L\)-fuzzy subcomplexes of \(\Delta\) is denoted by \(\calfc(\Delta,L)\).
\end{definition}

This condition implies that the membership value of a simplex is bounded by the values of all its faces. In particular, if \(\sigma = \langle v_0,\ldots,v_d \rangle \in \Delta\), then \(\mu(\sigma) \leq \bigwedge_{i=0}^d \mu(\langle v_i \rangle)\). 

\begin{example}
\label{exa:chroTDA}
In TDA, it is common to study the properties of a point cloud \(X = \{x_i \in \bbr^d \mid i=1,\dots,n\}\) by constructing a simplicial complex \(\Delta\) with vertex set \(X\). Now consider a chromatic dataset \(\mathcal{D} = (X,f)\) where \(f: X \to C=\{c_1,\dots,c_k\}\) is the labeling map. As discussed in \Cref{exa:fdlfuzzy}, \(f\) can be regarded as an \(L\)-fuzzy subset of \(X\), where \((L,\le)=\fdl(c_1,\dots,c_k)\). Given a simplicial complex \(\Delta\) built on \(X\), this induces a chromatic \(L\)-fuzzy subcomplex \(\mu \in \calfc(\Delta,L)\) defined by
\[
\mu(\langle x_0,\dots,x_n\rangle) = \bigwedge_{j=0}^{n} f(x_j).
\]
The standard pipeline of chromatic TDA partitions \(X\) into subsets \(X_1,\dots,X_k\), where \(X_j = \{x_i \in X \mid f(x_i)=c_j\}\), builds simplicial complexes on each subset, and analyzes the relationships among them. Nevertheless, the chromatic \(L\)-fuzzy subcomplex \(\mu \in \calfc(\Delta,L)\) defined above allows us to study \(\mathcal{D} = (X,f)\) with just one combinatorial structure.
\end{example}

We now state some basic properties of \(L\)-fuzzy subcomplexes.

\begin{proposition}
\label{prop:subcomplexesproperties}
Let \(\Delta\) be a simplicial complex and \((L,\leq)\) a CDL.
\begin{enumerate}
\item Let \(\{\mu_i \mid i \in I\} \subseteq \calfc(\Delta,L)\). Then, both \(\mu_{\wedge} = \bigcap_{i\in I} \mu_i\) and \(\mu_{\vee} = \bigcup_{i\in I} \mu_i\) belong to \(\calfc(\Delta,L)\).
\item For any \(\mu \in \calfc(\Delta,L)\) and any \(\ell \in L\), the upper level set \(\mu^{\geq \ell} \subseteq \Delta\) is either empty or a crisp subcomplex of \(\Delta\).
\item For any \(\mu \in \calfc(\Delta,L)\), the support \(\mu^*\) is a crisp subcomplex of \(\Delta\).
\end{enumerate}
\end{proposition}

\begin{proof}
We prove each statement in turn.
\begin{enumerate}
\item Let \(\sigma_1 \subseteq \sigma_2\) in \(\Delta\). Since each \(\mu_i\) is an \(L\)-fuzzy subcomplex, we have 
\(\mu_{\vee}(\sigma_1) \geq \mu_i(\sigma_1) \geq \mu_i(\sigma_2) \geq \mu_{\wedge}(\sigma_2)\) for all \(i \in I\). On one hand this implies \(\mu_{\wedge}(\sigma_1) =\bigwedge_{i \in I} \mu_i(\sigma_1) \geq \mu_{\wedge}(\sigma_2)\), which proves \(\mu_{\wedge} \in \calfc(\Delta,L)\). On the other hand this implies \(\mu_{\vee}(\sigma_1) \geq  \bigvee_{i \in I} \mu_i(\sigma_2)=\mu_{\vee}(\sigma_2)\), which proves \(\mu_{\vee} \in \calfc(\Delta,L)\). 
\item If \(\mu(\sigma) \not \geq \ell\) for all \(\sigma \in \Delta\), then \(\mu^{\geq \ell} = \emptyset\). Otherwise, let \(\sigma_1 \in \mu^{\geq \ell}\), so \(\mu(\sigma_1) \geq \ell\) and take any face \(\sigma_2 \subseteq \sigma_1\). Since \(\mu \in \calfc(\Delta,L)\), we have \(\mu(\sigma_2) \geq \mu(\sigma_1) \geq \ell\), so \(\sigma_2 \in \mu^{\geq \ell}\). Consider now two simplices \(\sigma_1, \sigma_2 \in \mu^{\geq \ell}\). When \(\sigma_1 \cap \sigma_2\) is not empty, it is a common face in \(\Delta\) so \(\mu(\sigma_1 \cap \sigma_2) \geq \mu(\sigma_1) \geq \ell\) and therefore \(\sigma_1 \cap \sigma_2 \in \mu^{\geq \ell}\). Hence \(\mu^{\geq \ell}\) is a crisp subcomplex of \(\Delta\).
\item This follows directly from the two previous items because \(\mu^*=\mu^{>0}=\bigcup_{\ell>0} \mu^{\geq \ell}\).
\qedhere
\end{enumerate}
\end{proof}

Repeating the argument of \Cref{sec:lfuzzysubmodules} for \(L\)-fuzzy submodules, the contravariant functor \(\cut(\mu): L \to \calp(\Delta)\) associated with \(\mu \in \calfc(\Delta,L)\) factors through the category \((\mathrm{Sub}(\Delta),\subseteq)\) of subcomplexes of \(\Delta\), with morphisms restricted to inclusions. That is, there exists a functor \(g: L \to \mathrm{Sub}(\Delta)\) such that \(\cut(\mu)= i \circ g\), where \(i: \mathrm{Sub}(\Delta) \to \calp(\Delta)\) is the inclusion functor. In particular, \(\cut(\mu)\) may be identified with \(g\), and thus regarded as a functor \(\cut(\mu): L \to \sub(\Delta)\).

We have discussed in \Cref{exa:chroTDA} how \(L\)-fuzzy subcomplexes can be applied to model chromatic datasets. Now, we study the relation between \(L\)-fuzzy subcomplexes and \emph{filtrations}, one of the most studied objects in TDA.

\begin{definition}[Filtrations and decreasing filtrations]
\label{def:filtration}
Let \((P, \leq)\) be a poset and \((\SpCpx,\subset)\) the category of simplicial complexes (not necessarily finite) with morphisms restricted to inclusions \footnote{Notice that one could also consider the category of regular complexes. Here we restrict to the category of simplicial complexes for simplicity and in order to be consistent with existing literature~\cite{spivak2009metric}.}. 
A \emph{filtration} over \(P\) is a covariant functor $F\colon P \rightarrow \SpCpx$, and a \emph{decreasing filtration} over \(P\) is a contravariant functor $G\colon P \rightarrow \SpCpx$.
\end{definition}

In other words, a filtration is an increasing sequence of simplicial complexes indexed by \(P\). Indeed, if $p_1\leq p_2$, it holds that \(F(p_1) \subseteq F(p_2)\). On the other hand, a decreasing filtration satisfies that \(G(p_1) \supseteq G(p_2)\) whenever \(p_1 \leq p_2\).
Given a filtration \(F\) over \(P\), let \(\Sigma_F=\bigcup_{p \in P}F(p)\) be the simplicial complex that contains the whole filtration. In the remainder of this section, we discuss that any such filtration can be modeled as an \(L\)-fuzzy subcomplex of \(\Sigma_F\).

\begin{definition}[Up-sets]
\label{def:upset}
Let \((P,\leq)\) be a poset. A subset \(S \subseteq P\) is called an \emph{up-set} if for all \(p,q \in P\), whenever \(p \in S\) and \(p \le q\), then \(q \in S\). The set of all up-sets of \(P\) is denoted by \(P_\uparrow\).
\end{definition}

The filters \(P^{\geq p}\) defined in \Cref{def:Lfilter} are examples of up-sets. In fact, the set \(\{P^{\geq p} \mid p \in P\}\) is contained in \(P_\uparrow\), but both sets do not coincide when \(P\) is not totally ordered.

\begin{proposition}
\label{prop:upsetcdl}
The poset \((P_\uparrow, \subset)\) is a CDL where the joins correspond to unions, meets correspond to intersection and the least and greatest elements are \(\emptyset\) and \(P\) respectively.
\end{proposition}

\begin{proof}
The set \(P_\uparrow\) is a subset of \(\calp(P)\), and it is proved in \cite[Theorem 10.24]{davey2002introduction} that \((\calp(P), \subset)\) is a CDL where the joins correspond to unions and meets to intersections. Then, it remains to prove that \(P_\uparrow\) is closed under arbitrary unions and intersections. Given an arbitrary collection of up-sets \(\{S_i \in P_\uparrow \mid i \in I\}\), we claim that \(\bigcup_{i \in I} S_i\) is an up-set. Indeed, if \(p \in \bigcup_{i \in I}S_i\), then \(p \in S_i\) for some \(i \in I\). Since \(S_i\) is an upset, for any \(q \in P\) with \(p \le q\) we have \(q \in S_i\), and therefore \(q \in \bigcup_{i \in I}S_i\). We claim that \(\bigcap_{i \in I}S_i\) is an up-set too. Indeed, if \(p \in \bigcup_{i \in I}S_i\), then \(p \in S_i\) for all \(i \in I\). Since each \(S_i\) is an upset, for any \(q \in P\) with \(p \le q\) we have \(q \in S_i\), and therefore \(q \in \bigcap_{i \in I}S_i\). The up-sets \(\emptyset\) and \(P\) are trivially the least and greatest elements of \(P_\uparrow\).   
\end{proof}

We now prove a result connecting decreasing filtrations over a CDL and \(L\)-fuzzy subcomplexes.

\begin{proposition}\label{prop:filtration-fuzzy-characterisation}
Let \((L,\leq)\) be a CDL, let \(M\colon L\rightarrow \SpCpx\) be a decreasing filtration and denote $\Sigma_M=\bigcup_{\ell\in L}M(\ell)$. 
There exists \(\mu \in \calfc(\Sigma_M,L)\) such that $M(\ell) = \mu^{\geq \ell}$ for all $\ell \in L$ if and only if $M(\bigvee S)=\bigcap_{\ell\in S} M(\ell)$ for any subset $S \subset L$.
\end{proposition}
\begin{proof}
Assume that there exists \(\mu \in \calfc(\Sigma_M,L)\) such that $M(\ell) = \mu^{\geq \ell}$ for all $\ell \in L$. 
By \Cref{prop:subsetsproperties}, for any \(S \subseteq L\) we have \(\mu^{\geq \bigvee S} = \bigcap_{\ell \in S}\mu^{\geq \ell}\), and by hypothesis this is equivalent to \(M(\bigvee S) = \bigcap_{\ell \in S} M(\ell)\).

Conversely, assume that \(M(\bigvee S) = \bigcap_{\ell \in S} M(\ell)\) for every subset \(S \subseteq L\).  
For each simplex \(\sigma \in \Sigma_M\), define the set \(L_{\sigma} = \{c \in L \mid \sigma \in M(c)\}\), and consider the map \(\mu_M : \Sigma_M \to L\) given by \(\mu_M(\sigma) = \bigvee L_{\sigma}\). Given two simplices $\sigma_1 \subset \sigma_2$ from $\Sigma_M$, we have that \(L_{\sigma_1} \supseteq L_{\sigma_2}\) and therefore $\mu_M(\sigma_1) \geq \mu_M(\sigma_2)$. Thus, $\mu_M \in \calfc(\Sigma_M,L)$. It remains to show that $M(\ell) = \mu_M^{\geq \ell}$ for all $\ell \in L$. 

Given $\sigma \in M(\ell)$, we have that \(\ell \in L_{\sigma}\). Then $\mu_M(\sigma)\geq \ell$, which implies $\sigma \in \mu_M^{\geq \ell}$.
Given $\sigma \in \mu_M^{\geq \ell}$, we have that $\mu_M(\sigma)\geq \ell$. Now, by hypothesis, 
\(\sigma \in \bigcap_{c \in L_{\sigma}} M(c)=M(\bigvee L_{\sigma})=M(\mu_M(\sigma))\). Since \(M\) is a decreasing filtration and $\mu_M(\sigma)\geq \ell$, we have that \(M(\mu_M(\sigma))\subseteq M(\ell)\) and therefore \(\sigma \in M(\ell)\), completing the proof.
\end{proof}

Consider again a filtration over a poset $F\colon P\rightarrow \SpCpx$.
From $F$, we can define a decreasing filtration over the CDL of up-sets $M_F\colon P_\uparrow \rightarrow \SpCpx$ given by $M_F(A) = \bigcap_{p \in A}F(p)$ for all $A \in P_\uparrow$.
It is indeed a decreasing filtration because, for any pair of up-sets $A,B \in P_\uparrow$ with $A \subset B$, it holds that $M_F(A) \supset M_F(B)$.
In particular, it follows that $M_F(P ^{\geq p} ) = F(p)$ for all $p \in P$.
This construction allows to factor the filtration $F\colon P \rightarrow \SpCpx$ as the composition of two contravariant functors $F=M_F \circ \iota $, where $\iota \colon (P, \leq)\rightarrow (P_\uparrow, \subset)$ is given by the assignment $p\mapsto P^{\geq p}$.
Finally, note that $\Sigma_F = \bigcup_{p \in P} F(p)= \bigcup_{A \in P_\uparrow} M_F(A) = \Sigma_{M_F}$.
Then, we have the following result.

\begin{proposition}\label{prop:any-filtration-is-fuzzy}
Let $F\colon P\rightarrow \SpCpx$ be a filtration and let $\Sigma_F=\bigcup_{p \in P}F(p)$. 
There exists \(\mu_F \in \calfc(\Sigma_F,P_\uparrow)\) such that 
$\mu_F^{\supseteq A} =M_F(A)$ for all $A \in P_\uparrow$. 
\end{proposition}
\begin{proof}
First, we prove that $M_F(\bigvee S)=\bigcap_{A\in S} M_F(A)$ for any subset $S \subset P_\uparrow$, recalling that in the CDL \((P_\uparrow,\subset)\) the join \(\bigvee S\) is equal to \(\bigcup_{A \in S} A\). This is quite direct, because:
\[
\bigcap_{A \in S} M_F(A) = \bigcap_{A \in S}\bigcap_{p \in A}F(p) = \bigcap_{p \in \bigvee S}F(p) = M_F\bigg(\bigvee S\bigg).
\]
Thus, by \Cref{prop:filtration-fuzzy-characterisation} there exists \(\mu_F \in \calfc(\Sigma_{M_F},P_\uparrow) = \calfc(\Sigma_{F},P_\uparrow)\) such that $\mu_F^{\supseteq A} =M_F(A)$ for all $A \in P_\uparrow$.
\end{proof}
In particular, the \(L\)-fuzzy subcomplex \(\mu_F \in \calfc(\Sigma_{F},P_\uparrow)\) satisfies that $\mu_F^{\supseteq P^{\geq p}}=F(p)$ for all $p \in P$. 

We have seen that many constructions that arise naturally in TDA can be interpreted as \(L\)-fuzzy subcomplexes. In the next section, we introduce a new approach to defining a homology theory on these objects.

\section{Simplicial homology and L-Fuzzy simplicial homology}
\label{sec:fuzhom}

Simplicial homology is an algebraic tool that assigns to a simplicial complex a sequence of \(\bbd\)-modules (being \(\bbd\) a PID), capturing information about the connectivity and overall shape of the underlying topological space. In this section, we  introduce our proposed definition of \(L\)-fuzzy simplicial homology, which assigns to an \(L\)-fuzzy subcomplex a sequence of \(L\)-fuzzy submodules that complement the topological information given by homology \(\bbd\)-modules of simplicial complexes.

\begin{definition}[Oriented simplices \cite{croom2012basic}]
\label{def:oriented}
Let \(\Delta\) be a simplicial complex, and consider an order in its vertex set. Given a \(d\)-simplex \(\sigma =\langle v_0,\dots,v_d\rangle\) in \(\Delta\), any ordering of its vertices is called an \emph{oriented simplex} (for example \([v_0, \dots, v_d]\)). The oriented simplex \([v_0, \dots, v_d]\) is said to be \emph{positively oriented} if the vertices can be ordered with an even permutation or \emph{negatively oriented} otherwise. 
\end{definition}

For any \(d\)-simplex, there are \((d\!+\!1)!\) possible oriented simplices, one for each permutation of its vertices. This orientation induces an equivalence relation on the set of oriented simplices:  
\[
[v_0, \dots, v_d] \sim [v_{\pi(0)}, \dots, v_{\pi(d)}],  \text{if } \pi \text{ is even}.
\]
Hence, the oriented simplices of a given simplex form two equivalence classes, corresponding to the two possible orientations (positive and negative). 
For example, given a 2-simplex \(\langle v_0, v_1, v_2 \rangle\), we have 
\([v_0, v_1, v_2]  \not \sim [v_1, v_0, v_2]\) but
\([v_0, v_1, v_2] \sim [v_1, v_2, v_0]\).
The \(0\)-simplices are always positively oriented.

\begin{definition}[Crisp module of \(d\)-chains]
\label{def:chain}
Let \(\Delta\) be a simplicial complex and let \(\bbd\) be a PID. For each integer \(d \geq 0\), we define the \(\bbd\)-module of \(d\)-chains \(\Co_d\) of \(\Delta\) as the free \(\bbd\)-module generated by the positively oriented \(d\)-simplices of \(\Delta\). 
Simplices that are oriented negatively are embedded in \(\Co_d\) by setting 
\(
[v_{\pi(0)}, \dots, v_{\pi(d)}] \coloneqq -[v_0, \dots, v_d]
\) whenever \(\pi\) is odd.
When \(\Delta_d = \emptyset\), then \(\Co_d = \{0\}\) is the trivial \(\bbd\)-module. Similarly, we define \(\Co_d = \{0\}\) for \(d<0\). The elements of \(\Co_d\) are called \emph{\(d\)-chains}.
\end{definition}

Assume that \(\Delta_d\) contains \(n_d\) \(d\)-simplices. Then the set \(\E_d = \{\sigma^{d}_{1}, \dots, \sigma^{d}_{n_d}\}\) of positively oriented \(d\)-simplices of \(\Delta\) forms a basis for \(\Co_d\), and each \(d\)-chain \(c \in C_d\) can be written uniquely as a linear combination:
\[
c = \sum_{i=1}^{n_d} c_i \sigma^{d}_{i}, \quad \text{with } c_i \in \bbd.
\]
The vector of coefficients \(c^{\Delta} = (c_1, \dots, c_{n_d})' \in \bbd^{n_d}\) 
(where \( (\;)'\)  denotes transpose)
represents the coordinates of \(c\) with respect to the basis \(\E_d\).

\begin{definition}[Boundary operator, \(d\)-cycles and \(d\)-boundaries]
\label{def:boundary}
Let \(\Delta\) be a simplicial complex, \(\bbd\) a PID and let \(\Co_d\) be the \(\bbd\)-module of \(d\)-chains of \(\Delta\). 
\begin{itemize}
\item The \emph{\(d\)-th boundary operator} is the homomorphism of \(\bbd\)-modules \(\partial_d : \Co_d \to \Co_{d-1}\) that acts on each \(\sigma^{d}_{i} \in E^{\Delta}_{d}\) by:
\[
\partial_d(\sigma^{d}_i)=\partial_d([v_{i_0},\ldots,v_{i_d}]) := \sum_{j=0}^d (-1)^j [v_{i_0}, \dots, \hat{v}_{i_j}, \dots, v_{i_d}],
\]
where \(\hat{v}_i\) indicates that the vertex \(v_i\) is omitted. The \((d\!-\!1)\)-chain \(\partial_d(c) \in \Co_{d-1}\) is called the \emph{boundary} of \(c\).
\item The set \(\Zo_d = \ker(\partial_d) \subseteq \Co_d\) is called the \emph{submodule of \(d\)-cycles}, and its elements are called \emph{\(d\)-cycles}.
\item The set \(\Bo_d = \im(\partial_{d+1}) \subseteq \Co_d\) is called the \emph{submodule of \(d\)-boundaries}, and its elements are called \emph{\(d\)-boundaries}.
\end{itemize}
\end{definition}

This homomorphism maps each positively oriented \(d\)-simplex to a signed sum of its \((d\!-\!1)\)-dimensional faces, and extends linearly to all of \(\Co_d\). A \(d\)-cycle is a \(d\)-chain with trivial boundary, and a \(d\)-boundary is a \(d\)-chain that arises as the boundary of a \((d\!+\!1)\)-chain. The boundary operators satisfy that \(\partial_d \circ \partial_{d+1} \equiv 0\) for all 
\(d \in \bbz\)
(see \cite[Chapter 7]{rotman2013introduction} for a proof), meaning that \(\Bo_d \subseteq \Zo_d\) and the following sequence forms a chain complex of \(\bbd\)-modules. 
\[
\mathcal{C}(\Delta): \cdots \longrightarrow \Co_{d+1}\xrightarrow{\partial_{d+1}} \Co_d \xrightarrow{\partial_d} \Co_{d-1} \longrightarrow \cdots
\]

\begin{definition}[Crisp simplicial homology]
\label{def:homology}
Let \(\Delta\) be a simplicial complex and \(\bbd\) a PID. The \emph{\(d\)-homology of \(\Delta\)} is defined as the quotient \(\bbd\)-module:
\[
\Ho_d = \Zo_d/\Bo_d = \ker(\partial_d)/\operatorname{im}(\partial_{d+1}).
\]
The elements in \(\Ho_d\) are called \emph{\(d\)-homology classes}.
\end{definition}

Given a \(d\)-cycle \(h \in \Zo_d\), the coset \([h] = h + \Bo_d \in \Ho_d\) is called the \emph{\(d\)-homology class of \(h\)}. The \(\bbd\)-module \(\Ho_d\) captures the \(d\)-dimensional topological features (or ``holes'') of the simplicial complex \(\Delta\). For example, each class in \(\Ho_0\) corresponds to a connected component of \(\Delta\). Each class in \(\Ho_1\) represents a loop, that is, a \(1\)-cycle which is not the boundary of any collection of \(2\)-simplices. More generally, \(\Ho_d\) detects \(d\)-dimensional voids in the simplicial complex. The homology of \(\Delta\) depends on the choice of \(\mathbb{D}\), but it is independent of the ordering of the vertices of \(\Delta\). To avoid ambiguity, we always specify the coefficient PID under consideration. 

In summary, given a simplicial complex \(\Delta\) we define the \(\bbd\)-module \(\Co_d\) of \(d\)-chains and obtain the submodules of \(d\)-cycles \(\Zo_d\) and \(d\)-boundaries \(\Bo_d\) as the kernel and image of the boundary operator \(\partial\), respectively. The \(d\)-homology module is then defined as the quotient \(\Ho_d = \Zo_d / \Bo_d\). 

Now, given an \(L\)-fuzzy subcomplex \(\mu \in \calfc(\Delta,L)\), we define an \(L\)-fuzzy submodule of \(d\)-chains \(\kappa_d \in \calfm(\Co_d,L)\), and then construct two \(L\)-fuzzy submodules \(\zeta_d \in \calfm(\Zo_d,L)\) and \(\beta_d \in \calfm(\Bo_d,L)\) via the kernel and image of the boundary operator. Our proposed definition of
\(L\)-fuzzy simplicial homology is then given by the quotient \(\eta_d = \zeta_d / \beta_d\). 
Observe that our proposed definition is quite natural, as it mirrors this pipeline applying the definitions and results given in \Cref{sec:lfuzzysubmodules,sec:lfuzzysubcomplexes}.
We now explain in detail the definition and the computation of such \(L\)-fuzzy submodules.

Recall that an \(L\)-fuzzy subcomplex \(\mu \in \calfc(\Delta,L)\) is a map on the \(d\)-simplices without orientation. However, abusing of notation, for a positively oriented \(d\)-simplex \(\sigma^{d}_{i} = [v_0,\ldots,v_d] \in E^{\Delta}_d\) we write \(\mu(\sigma^{d}_i)\) instead of \(\mu(\langle v_0,\ldots,v_d \rangle)\). 

\begin{definition}[\(L\)-fuzzy \(d\)-chains, \(d\)-cycles, and \(d\)-boundaries]
\label{def:chainsubmodules}
Let \(\Delta\) be a simplicial complex, let \(\Co_d\) be the \(\bbd\)-module of \(d\)-chains of \(\Delta\)
and let \(\mu \in \calfc(\Delta,L)\). 
\begin{itemize}
\item The map \(\delta_d = \bigcup_{i=1}^{n_d} \{\sigma^{d}_i\}_{\mu(\sigma^{d}_i)} \in \calfp(\Co_d,L)\) is called the \emph{\(L\)-fuzzy subset of \(d\)-simplices of \(\mu\)}.
\item The map \(\kappa_d = \langle \delta_d \rangle \in \calfm(\Co_d,L)\) is called the \emph{\(L\)-fuzzy submodule of \(d\)-chains of \(\mu\)}.
\item The map \(\zeta_d = \kappa_d \cap \partial_d^{-1}(\{0\}_1) \in \calfm(\Co_d,L)\) is called the \emph{\(L\)-fuzzy submodule of \(d\)-cycles of \(\mu\)}.
\item The map \(\beta_d = \kappa_d \cap \partial_{d+1}(\kappa_{d+1}) \in \calfm(\Co_d,L)\) is called the \emph{\(L\)-fuzzy submodule of \(d\)-boundaries of \(\mu\)}.
\end{itemize}
\end{definition}

We claim in this definition that \(\kappa_d\), \(\zeta_d\), and \(\beta_d\) are \(L\)-fuzzy submodules of \(\Co_d\). Indeed, by \Cref{def:generatedsubmodule}, \(\langle \delta_d\rangle\) is always an \(L\)-fuzzy submodule; by \Cref{prop:extension_submodules}, the image and preimage of an \(L\)-fuzzy submodule under a module homomorphism are again \(L\)-fuzzy submodules; and by \Cref{prop:submodules_properties}, the intersection of two \(L\)-fuzzy submodules is also an \(L\)-fuzzy submodule. 
Applying \Cref{def:fuzzy_operations} to \(\partial_d^{-1}(\{0\}_1)\), it follows that \(\partial_d^{-1}(\{0\}_1)(c)=1\) if \(c \in \Zo_d\) and \(\partial_d^{-1}(\{0\}_1)(c)=0\) otherwise. Then, \(\zeta_d\) could be equivalently defined as \(\zeta_d = \kappa_d \cap (\Zo_d)_1\).

These \(L\)-fuzzy submodules play the roles of the crisp chain, cycle, and boundary modules, respectively: \(\kappa_d\) corresponds to \(\Co_d\), \(\zeta_d\) to \(\Zo_d\), and \(\beta_d\) to \(\Bo_d\). 
We now establish some basic properties of these \(L\)-fuzzy submodules.

\begin{proposition}
\label{prop:chains_submodules_properties}
Let \(\Delta\) be a simplicial complex, let \(\Co_d\) be the \(\bbd\)-module of \(d\)-chains of \(\Delta\)
and let \(\mu \in \calfc(\Delta,L)\). Then:
\begin{enumerate}[label=(\roman*)]
\item\label{item:d-chain-fuzzy} For any \(d\)-chain \(c = \sum_{i=1}^{n_d} c_i \sigma^{d}_i \in \Co_d\), we have
\(\kappa_d(c) = \bigwedge_{i=1,\ldots,n_d}^{c_i \neq 0} \delta_d(\sigma^{d}_i) = \bigwedge_{i=1,\ldots,n_d}^{c_i \neq 0}  \mu(\sigma^{d}_i)\).
\item\label{item:supportZB} \(\zeta_d^* \subseteq \Zo_d\) and \(\beta_d^* \subseteq \Bo_d\).
\item\label{item:beta-zeta-kappa} \(\beta_d \subseteq \zeta_d \subseteq \kappa_d\) and, for all \(z\in \Zo_d\), it holds that   \(\zeta_d(z)= \kappa_d(z)\).
\item \label{item:meet-prime} If \(\mu^* = \Delta\) and 0 is meet-prime in \(L\), then \(\kappa_d^* = \Co_d\), \(\zeta_d^* = \Zo_d\), and \(\beta_d^* = \Bo_d\).
\end{enumerate}
\end{proposition}

\begin{proof}
We prove all the statements one by one.
\begin{enumerate}[label=(\roman*)]
\item Since \(\E_d=\{\sigma^{d}_{1}, \dots, \sigma^{d}_{n_d}\}\) is a basis of \(\Co_d\), it is linearly independent and this follows from \Cref{cor:singletons_generating}.
\item To prove \(\zeta_d^* \subseteq \Zo_d\), let \(c\in \Co_d\) be such that \(c \notin \Zo_d\). Then \(\partial_d(c) \neq 0\), so \(\zeta_d(c) = \kappa_d(c) \wedge \partial_d^{-1}(\{0\}_1)(c) = \kappa_d(c) \wedge (\{0\}_1)(\partial_d(c)) = 0.\) Thus, $c \notin \zeta_d^*$. To prove \(\beta_d^* \subseteq \Bo_d\), suppose that \(c \notin \Bo_d = \im(\partial_{d+1})\). Then, 
\(\partial_{d+1}(\kappa_{d+1})(c) = \bigvee \{\kappa_{d+1}(a) \mid \partial_{d+1}(a)=c\} = \bigvee \emptyset = 0\) and $\beta_d(c)=0$. Thus, $c \notin \beta_d^*$.
\item By definition, we already have $\zeta_d, \beta_d \subseteq \kappa_d$. To prove $\beta_d \subseteq \zeta_d$, consider any $c \in \Co_d$. If $c \notin \Bo_d$, then $\beta_d(c) = 0 \leq \zeta_d(c)$. If  on the contrary $c \in \Bo_d$, then $c \in \Zo_d$ as well. Then, \(\zeta_d(c) = \kappa_d(c) \wedge \partial_d^{-1}(0_1)(c) = \kappa_d(c)\), and
\(\beta_d(c) = \kappa_d(c) \wedge \partial_{d+1}(\kappa_{d+1})(c) \leq \kappa_d(c) = \zeta_d(c)\).
\item From \ref{item:d-chain-fuzzy}, we have that \(\kappa_d(c) = \bigwedge \{\mu(\sigma^{d}_{i}) \mid i = 1, \dots, n_d,\ c_i \neq 0\}\). Since \(\mu^* = \Delta\), we have \(\mu(\sigma^{d}_{i}) > 0\) for all \(i = 1, \dots, n_d\). As \(0\) is meet-prime in \(L\), the meet of finitely many non-zero elements is also non-zero. Therefore, \(\kappa_d(c) > 0\) for all \(c \in C_d\), and consequently \(\kappa_d^* = C_d\).

The inclusions \(\zeta_d^* \subseteq \Zo_d\) and \(\beta_d^* \subseteq \Bo_d\) were already established in~\ref{item:supportZB}. We now prove the reverse inclusions. Let \(c \in \Zo_d\). By~\ref{item:beta-zeta-kappa}, \(\zeta_d(c) = \kappa_d(c) > 0\), hence \(c \in \zeta_d^*\) and \(\Zo_d \subseteq \zeta_d^*\). Similarly, if \(c \in \Bo_d\), then \(c = \partial_{d+1}(c')\) for some \(c' \in \Co_{d+1}\). Since \(\kappa_{d+1}^* = \Co_{d+1}\), we have \(\kappa_{d+1}(c') > 0\), and \(\partial_{d+1}(\kappa_{d+1})(c) = \bigvee \{\kappa_{d+1}(z) \mid \partial_{d+1}(z)=c\} \geq \kappa_{d+1}(c') > 0\). Moreover, \(\kappa_d(c) > 0\), and since \(0\) is meet-prime in \(L\), it follows that \(\beta_d(c) = \kappa_d(c) \wedge \partial_{d+1}(\kappa_{d+1})(c) > 0\). Thus, \(c \in \beta_d^*\) and \(\Bo_d \subseteq \beta_d^*\). \qedhere
\end{enumerate}
\end{proof}

We are now ready to extend the classical notion of homology to the \(L\)-fuzzy setting.
\begin{definition}[L-fuzzy simplicial homology]
\label{def:eta}
Let $\Delta$ be a simplicial complex and let $\mu \in \calfc(\Delta,L)$.  
We define the \emph{\(L\)-fuzzy \(d\)-homology of \(\mu\)} as the quotient:
\[
\eta_d = \zeta_d/\beta_d
\in \calfm(\langle\zeta_d^*\rangle/\langle\beta^*_d\rangle,L).
\]
\end{definition}

That is, \(\eta_d\) is obtained by applying \Cref{def:lfuzzyquotient} to the \(L\)-fuzzy submodules \(\zeta_d\) and \(\beta_d\). Since this construction depends on their supports, it is natural to analyze the role of simplices \(\sigma \in \Delta\) such that \(\mu(\sigma)=0\). Let \(\Delta'=\mu^*\) denote the support, which is a crisp subcomplex of \(\Delta\) as proved in \Cref{prop:subcomplexesproperties}, and let \(\mu' = \mu|_{\Delta'} \in \calfc(\Delta',L)\). Applying \Cref{def:chain,def:boundary,def:homology} to \(\Delta'\) we can define the \(\bbd\)-modules \(\Co'_d, \Zo'_d, \Bo'_d, \Ho'_d\), and applying \Cref{def:chainsubmodules,def:eta} to \(\mu'\) we can define the \(L\)-fuzzy submodules \(\delta'_d, \kappa'_d, \zeta'_d, \beta'_d, \eta'_d\).

\begin{proposition}
\label{prop:supportcomplex}
Let \(\Delta\) be a simplicial complex, let \(\mu \in \calfc(\Delta,L)\) and let \(\mu' = \mu|_{\Delta'} \in \calfc(\Delta',L)\). 
Then, \(\eta_d=\eta'_d\).
\end{proposition}

\begin{proof}
By construction, \(\Co_d\) is the \(\bbd\)-module generated by the positively oriented \(d\)-simplices of \(\Delta\) while \(\Co'_d\) is generated by those of \(\Delta'\), that is, those with \(\mu(\sigma^d_i)=\delta_d(\sigma^d_i)>0\). Hence, \(\Co'_d= \langle \delta_d^* \rangle \subseteq \Co_d\). 
In particular, \(\delta_d\) vanishes on \(\Co_d \setminus \Co'_d\), \(\delta'_d = \delta_d|_{\Co'_d}\) and \(\delta_d^* = (\delta'_d)^*\). The \(L\)-fuzzy submodules \(\kappa'_d,\zeta'_d,\beta'_d\) can be defined from \(\delta'_d\) as in \Cref{def:chainsubmodules} and it follows that 
\(\kappa'_d = \kappa_d|_{\Co'_d}\), \(\kappa_d^* = (\kappa'_d)^*\),
\(\zeta'_d = \zeta_d|_{\Co'_d}\), \(\zeta_d^* = (\zeta'_d)^*\),
\(\beta'_d = \beta_d|_{\Co'_d}\) and \(\beta_d^* = (\beta'_d)^*\).
Therefore, the maps \(\eta: \langle \zeta_d^* \rangle / \langle \beta_d^* \rangle \to L\) and \(\eta': \langle (\zeta'_d)^* \rangle / \langle (\beta'_d)^* \rangle \to L\) have the same domain and take the same values, so that \(\eta_d=\eta'_d\).
\end{proof}

This result shows that simplices of \(\Delta\) outside the support of \(\mu\) can be discarded without affecting the \(L\)-fuzzy \(d\)-homology submodule \(\eta_d\). This agrees with the interpretation of \(L\)-fuzzy subsets, since, as discussed in \Cref{sec:lfuzzysubsets}, the condition \(\mu(\sigma)=0\) means that \(\sigma\) does not belong to \(\mu\). Therefore, there is no loss of generality in assuming \(\mu^*=\Delta\).

We now verify that \(\eta_d\) extends crisp simplicial homology. To this end, we consider \Cref{def:eta} in the case where the \(L\)-fuzzy subcomplex \(\mu\) takes values in the CDL \((\{0,1\}, \leq)\).

\begin{proposition}
\label{prop:generalize}
Let \(\Delta\) be a simplicial complex and let \(\mu \in \calfc(\Delta,\{0,1\})\) such that \(\mu^*=\Delta\). Then, \(\eta_d \in \calfm(\Ho_d,\{0,1\})\) and \(\eta_d([h])=1\) for all \([h] \in \Ho_d\). 
\end{proposition}

\begin{proof}
Since \(\mu^*=\Delta\), then \(\mu'(\sigma)=1\) for all \(\sigma \in \Delta'\) and it follows that \(\kappa_d(c)=1\) for all \(c \in \Co_d\), \(\zeta_d(z)=1\) for all \(z \in \Zo_d\) and \(\beta_d(b)=1\) for all \(b \in \Bo_d\). Then, the domain of \(\eta_d\) is \(\langle\zeta_d^*\rangle/\langle\beta_d^*\rangle= \Zo_d/\Bo_d=\Ho_d\) and \(\eta_d([h])=1\) for all \([h] \in \Ho_d\). 
\end{proof}

This result shows that, when the CDL is \((\{0,1\},\leq)\), then \(\eta_d\) reduces to the constant map with value \(1\) on \(\Ho_d\). Under our interpretation, \(\eta_d([h])=1\) means that \([h]\) belongs to \(\eta_d\). Hence, \(\eta_d\) describes \(\Ho_d\) in terms of \(L\)-fuzzy subsets and therefore \(L\)-fuzzy simplicial homology generalizes crisp simplicial homology. 

\begin{corollary}
\label{cor:eta_value}
Let \(\Delta\) be a simplicial complex, let \((L,\leq)\) be a CDL such that \(0 \in L\) is meet-prime and let \(\mu \in \calfc(\Delta,L)\) such that  \(\mu^* = \Delta\). If \(\Ho_d\) is the \(d\)-homology \(\bbd\)-module of \(\Delta\), then \(\eta_d \in \calfm(\Ho_d,L)\) and, for each \([h] \in \Ho_d\),
\[
\eta_d([h]) =
\bigvee \{ \zeta_d(z) \mid z \in [h] \}
=\bigvee \{ \kappa_d(z) \mid z \in [h]\}.
\]
In particular, $\eta_d^*=\Ho_d$.
\end{corollary}

\begin{proof}
Since \(\mu^* = \Delta\) and \(0\) is meet-prime in \(L\), we know by \cref{item:meet-prime} of \Cref{prop:chains_submodules_properties} that \(\zeta_d^* = \Zo_d\) and \(\beta_d^* = \Bo_d\). Since \(\Zo_d\) and \(\Bo_d\) are submodules of \(\Co_d\), it follows that \(\langle \zeta_d^* \rangle = \zeta_d^*= \Zo_d\) and \(\langle \beta_d^* \rangle = \beta_d^*= \Bo_d\). Thus, the domain of \(\eta_d\) is \(\langle\zeta_d^*\rangle/\langle\beta_d^*\rangle= \Zo_d/\Bo_d=\Ho_d\). 
Now, by \cref{item:beta-zeta-kappa} of \Cref{prop:chains_submodules_properties} we know that \(\zeta_d(z)=\kappa_d(z)\) for any \(d\)-cycle \(z \in \Zo_d\).
Hence, applying \Cref{def:lfuzzyquotient} to \(\eta_d = \zeta_d/\beta_d\), we obtain \(\eta_d([h]) = \bigvee \{ \zeta_d(z) \mid z \in [h] \} = \bigvee \{ \kappa_d(z) \mid z \in [h] \}\). Moreover, since \(\zeta_d(z)>0\) for all \(z \in \Zo_d\), it follows that \(\eta_d([h])>0\) for all \([h] \in \Ho_d\).
\end{proof}

If \(0\) is not meet-prime in \(L\), the \(L\)-fuzzy \(d\)-homology submodule \(\eta_d\in \calfm(\langle\zeta_d^*\rangle/\langle\beta^*_d\rangle,L)\) is still well defined, but it is not necessarily true that \(\langle\zeta_d^*\rangle/\langle\beta^*_d\rangle = \Ho_d\), because \(\langle \zeta_d^*\rangle\) may be strictly contained in \(\Zo_d\).
From now on, we always assume that \(\mu^*=\Delta\) and \(0\) is meet-prime in \(L\), so that \(\eta_d \in \calfm(\Ho_d,L)\).

\begin{remark}
\label{rem:Lsets}
Consider the sets \(L(\delta_d),L(\kappa_d),L(\eta_d)  \subseteq L\). 
The set \(L(\delta_d)\setminus\{0\}\) contains all the \(L\)-fuzzy values that the \(d\)-simplices in \(\Delta_d\) can take. 
The set \(L(\kappa_d)\) contains all the \(L\)-fuzzy values that the \(d\)-chains in \(\Co_d\) can take. 
By \Cref{prop:chains_submodules_properties}, it follows that \(L(\kappa_d) = \{\bigwedge S \mid S \subseteq L(\delta_d)\setminus \{0\}\}\).
The set \(L(\eta_d)\) contains all the \(L\)-fuzzy values that the \(d\)-homology classes in \(\Ho_d\) can take. 
Now, by \Cref{cor:eta_value}, it follows that \(L(\eta_d) \subseteq \{\bigvee S \mid S \subseteq L(\kappa_d)\}\).
Note that since \(\Delta_d\) is finite, these three sets are finite too.
When \((L,\leq)\) is totally ordered, we have \(L(\kappa_d) = (L(\delta_d)\setminus\{0\})\cup \{1\}\) and \(L(\eta_d) \subseteq L(\kappa_d)\).
\end{remark}

The definition of \(\eta_d([h])\) given in \Cref{cor:eta_value} is not always practical, as it requires computing \(\kappa_d(z)\) for every \(d\)-cycle \(z \in [h]\). Depending on the chosen PID \(\bbd\), the set \([h] = h + \Bo_d\) may even be infinite. For this reason, we develop in the next section an alternative method for computing \(\eta_d([h])\), which might be more convenient in specific cases. Additionally, we develop a method to compute the crisp submodule \(\eta_d^{\geq \ell}\) for any \(\ell \in L\).

\section{Computation of simplicial homology and L-fuzzy simplicial homology}
\label{sec:computation}

We start this section computing the homology \(\bbd\)-modules of a simplicial complex \(\Delta\).
Let \(t = \max\{d \geq 0 \mid \Delta_t \neq \emptyset\}\) be the dimension of \(\Delta\). For \(d<0\) and \(d>t\) , the \(\bbd\)-modules \(\Co_d\) and \(\Ho_d\) are trivial. For each \(d = 0,\dots,t\), we know from \Cref{thm:modules_structure} that the structure of \(\Ho_d\) is completely determined by its Betti number and torsion coefficients. To compute
them, we focus on the following subsequence of the chain complex \(\calc(\Delta)\):
\[
\{0\}=\Co_{t+1} \xrightarrow{} \Co_t \xrightarrow{} \dots \xrightarrow{} 
\Co_{d+1} \xrightarrow{\partial_{d+1}} \Co_d \xrightarrow{\partial_d} \Co_{d-1}
\xrightarrow{} \dots \xrightarrow{} \Co_0 \xrightarrow{} \Co_{-1} = \{0\}.
\]

For each \(d = 0,\dots,t\), the \(\bbd\)-module \(\Co_d\) is finitely generated by the basis \(E^{\Delta}_{d}\) and is torsion-free. Since the boundary operator \(\partial_d\) is linear, it can be represented by a matrix \(M_d \in \bbd^{n_{d-1} \times n_d}\), whose \(i\)-th column is the coordinate vector of \(\partial_d(\sigma^d_i)\) with respect to the basis \(E^{\Delta}_{d-1}\). With this notation, for any \(d\)-chain \(a \in \Co_d\) we have \(M_d\, a^{\Delta} = (\partial_d(a))^{\Delta}\).
Since \(\Co_{t+1}=\{0\}\) and \(\Co_{-1}=\{0\}\), the boundary operator \(\partial_0\) is represented by the zero matrix \(M_0 \in \bbd^{1 \times n_0}\) and the boundary operator \(\partial_{t+1}\) is represented by the zero matrix \(M_{t+1} \in \bbd^{n_t \times 1}\).

The identity \(\partial_d \circ \partial_{d+1} = 0\) implies that \(M_d \, M_{d+1} = 0\). Thus, the chain complex \(\calc(\Delta)\) can be encoded by a sequence of matrices \(M_0,\dots,M_{t+1}\) satisfying \(M_d \, M_{d+1} = 0\) for all \(d=0,\dots,t\). The Betti numbers and torsion coefficients of each \(\Ho_d\) can be fully determined from these matrices applying the following results.

\begin{theorem}[Smith normal form \cite{jacobson2012basic}]
\label{thm:snf}
Let \(\bbd\) be a PID, and let \(A \in \bbd^{m \times n}\) be an \(m \times n\) matrix with entries in \(\bbd\). Then there exist invertible matrices \(P \in \bbd^{m\times m}\) and \(Q \in \bbd^{n\times n}\) such that \(P\, A \, Q = D\), where \(D \in \bbd^{m \times n}\) is a diagonal matrix of the form \(D = \operatorname{diag}(d_1, \dots, d_r, 0, \dots, 0)\),
where \(r = \rank(A)\), each \(d_i \in \bbd \setminus \{0\}\), and \(d_1 \mid d_2 \mid \cdots \mid d_r\). The matrix \(D\) is called the \emph{Smith normal form of \(A\)}, and the elements \(d_1, \dots, d_r\) are called the \emph{invariant factors of \(A\)}.
\end{theorem}

\begin{proposition} 
\label{prop:diagonals-product-zero} 
Let \(r_d\) denote the rank of matrix \(M_d\). For each \(d=0,\dots,t+1\), there exists a basis \(E^H_d\) of \(\Co_d\) together with invertible change-of-basis matrices \(\mhd_d\) and \(\mdh_d = (\mhd_d)^{-1}\) such that \(D_d \coloneqq \mhd_d M_d \mdh_d\) has the form \(D_d= [0 \mid D'_d]\), where the first \(r_{d+1}\) columns are zero and \(D'_d\) is a diagonal matrix of rank \(r_d\). Moreover, these matrices satisfy \(D_d \, D_{d+1} = 0\) for each \(d=0,\dots,t\). \end{proposition}

\begin{proof}
We construct the matrices \(D_0,\dots,D_{t+1}\) and the corresponding change-of-basis matrices iteratively from the highest dimension downward.  
The base case is for \(d=t+1\). The chain group \(\Co_{t+1} = \{0\}\) has zero boundary, so \(M_{t+1} \in \bbd^{n_t \times 1}\) is already the zero matrix. We define
\[
D_{t+1} = M_{t+1}, \quad P_t = I_{n_t \times n_t}, \quad Q_{t+1} = I_{1 \times 1},
\] 
so that \(P_t \, M_{t+1} \, Q_{t+1} = D_{t+1}\). Its rank is \(r_{t+1} = 0\), and it trivially satisfies the required block form.

Assume by induction that we already have invertible matrices \(P_d \in \bbd^{n_d \times n_d}\) and \(Q_{d+1} \in \bbd^{n_{d+1}\times n_{d+1}}\) such that \(D_{d+1}=P_d M_{d+1}Q_{d+1}\) has been put into the required block form:
\[
D_{d+1} = 
\begin{pmatrix} 0_{n_{d}\times r_{d+2}} & D'_{n_{d} \times (n_{d+1} - r_{d+2})} \end{pmatrix} =
\begin{pmatrix}
0_{r_{d+1} \times r_{d+2}} & \text{diag}_{r_{d+1}\times r_{d+1}} & 0_{r_{d+1} \times (n_{d+1}-r_{d+2}-r_{d+1})} \\
0_{(n_{d}-r_{d+1}) \times r_{d+2}} & 0_{(n_{d}-r_{d+1}) \times r_{d+1}} & 0_{(n_{d}-r_{d+1})\times (n_{d+1}-r_{d+2}-r_{d+1})}
\end{pmatrix}.
\]
We now want to find invertible matrices \(P_{d-1}\) and \(Q_d\) such that \(D_{d} = P_{d-1}\, M_d \, Q_d\) has the required form and \(D_d\, D_{d+1}=0\).  Define \(N_d = M_d \, P_d^{-1} \in \bbd^{n_{d-1} \times n_d}\).
This matrix satisfies that
\[
N_d \, D_{d+1} =
M_d \, P_d^{-1} \, \overbrace{P_d \, M_{d+1} \, Q_{d+1}}^{D_{d+1}} = 
\underbrace{M_d M_{d+1}}_{0} \, Q_{d+1} = 0.
\] 
Since \(D_{d+1}\) has the block form described above and \(N_d \, D_{d+1} = 0\), it follows that the first \(r_{d+1}\) columns of \(N_d\) are zero. We can then split \(N_d\) into blocks:
\[
N_d = \begin{pmatrix} 0_{n_{d-1}\times r_{d+1}} & N'_d \end{pmatrix}, \quad N'_d \in \bbd^{n_{d-1} \times (n_d - r_{d+1})}.
\]

By \Cref{thm:snf}, there exist invertible matrices \(P_{d-1} \in \bbd^{n_{d-1} \times n_{d-1}}\), \(Q'_d \in \bbd^{(n_d - r_{d+1}) \times (n_d - r_{d+1})}\) and a diagonal matrix \(D'_d\) of rank \(r_d\) (called the Smith normal form of \(N'_d\)) such that
\[
P_{d-1} \, N'_d \, Q'_d = D'_d.
\]
We extend the matrix \(Q'_d\) to 
\[
Q_d = \begin{pmatrix} I_{r_{d+1}} & 0 \\ 0 & Q'_d \end{pmatrix} \in \bbd^{n_d \times n_d},
\]
which is invertible. Then, we define \(D_d\) as:
\[
D_d \coloneqq P_{d-1} \, N_d \, Q_d = P_{d-1} \, \begin{pmatrix} 0_{n_{d-1}\times r_{d+1}} & N'_d \end{pmatrix} \, \begin{pmatrix} I_{r_{d+1}} & 0 \\ 0 & Q'_d \end{pmatrix} = \begin{pmatrix} 0_{n_{d-1}\times r_{d+1}} & D'_d \end{pmatrix},
\]
which has the required form. To verify that \(D_d \, D_{d+1}=0\), we check first that \(Q^{-1}_d D_{d+1} = D_{d+1}\). Indeed,
\[
Q^{-1}_{d} \, D_{d+1}=
\begin{pmatrix} I_{r_{d+1}} & 0 \\ 0 & (Q'_{d})^{-1} \end{pmatrix}
\begin{pmatrix}
0 & \mathrm{diag}_{r_{d+1} \times r_{d+1}} & 0 \\
0 & 0 & 0
\end{pmatrix}
=
D_{d+1}.
\] 
Combining everything, we obtain
\begin{align*}
P_{d-1}\cdot N_{d}\cdot D_{d+1} = & \; 
P_{d-1} \cdot 
\overbrace{M_{d} \cdot P^{-1}_{d}}^{N_d}\cdot 
\overbrace{Q_{d} \cdot Q^{-1}_{d}}^{I} \cdot 
\overbrace{P_{d} \cdot M_{d+1} \cdot Q_{d+1}}^{D_{d+1}}
 = 0, \\
D_{d}\cdot D_{d+1} = & 
\underbrace{P_{d-1} \cdot M_{d} \cdot P^{-1}_{d} \cdot Q_{d}}_{D_d}  \cdot 
\underbrace{Q^{-1}_{d} \cdot P_{d} \cdot M_{d+1} \cdot Q_{d+1}}_{D_{d+1}}
 = 0.
\end{align*}
It is proved that \(D_d \, D_{d+1}=0\). We repeat this procedure for \(d = t, t-1, \dots, 0\), each time computing the Smith normal form of \(N'_d\) and extending the invertible matrix \(Q'_d\). At the last step we obtain \(D_0\), completing the transformation of all matrices \(M_0,\dots,M_{t+1}\) into \(D_0,\dots,D_{t+1}\) with the required block-diagonal form. Each transformation is realized by invertible matrices
\[
\mhd_d = Q_d^{-1} \, P_d, \qquad 
\mdh_d = (\mhd_d)^{-1} = P_d^{-1} Q_d,
\] 
so that if \(M_d\) represents \(\partial_d\) in the original bases \(E^\Delta_{d-1}, E^\Delta_d\), then \(D_d\) represents \(\partial_d\) in the new bases \(E^H_{d-1}, E^H_d\).
\end{proof}

\Cref{prop:diagonals-product-zero} provides an iterative method to find new bases for \(\Co_{0}, \dots,\Co_{t}\) such that the boundary operators are represented by new matrices \(D_0,\dots,D_{t+1}\) in a nearly diagonal form. These matrices divide
\(E^{H}_d\) into four groups of generators: 

\begin{itemize}
\item[(U)] \textbf{\(d\)-boundaries:} They are those whose column in \(D_d\) is null and whose row in \(D_{d+1}\) has a unit of \(\bbd\). We denote these generators as 
\(u^d_1,\ldots,u^d_{n_U}\).
\item[(T)] \textbf{\(d\)-cycles that generate torsion homology classes:} They are those whose column in \(D_d\) is null and whose row in \(D_{d+1}\) has a non-zero and non-unit element of \(\bbd\). We denote these generators as 
\(t^{d}_{1},\ldots,t^{d}_{n_T}\) and the numbers in their corresponding rows are 
\(a^d_{1},\ldots,a^d_{n_T}\), which satisfy \(a^d_{1} \mid \ldots \mid a^d_{n_T}\).
\item[(R)] \textbf{\(d\)-chains that are not cycles:} They are those whose column in \(D_d\) is not null. We denote these generators as \(r^d_{1},\ldots,r^d_{n_R}\).
\item[(F)] \textbf{\(d\)-cycles that generate free homology classes:} They are those whose column in \(D_d\) is null and whose row in \(D_{d+1}\) is also null. We denote this generators as
\(f^d_{1},\ldots,f^d_{n_F}\).
\end{itemize}

Clearly, \(n_d = n_{U} + n_{T} + n_{R} + n_{F}\). The reduction process sorts these generators in such a way that at the beginning go the generators of group U, then those of T, then those of R and finally those of F. This division of \(E^{H}_d\) into four groups induces a division of the change-of-basis matrix \(\mdh_d\) into four blocks 
\(\mdh_d = \left(U_d \mid T_d \mid R_d \mid F_d \right)\). Describing \(\Ho_d = \Zo_d /\Bo_d\) is now quite direct using \(E^{H}_d\), because we have:
\begin{align*}
\Co_d = & \langle u^d_1,\ldots,u^d_{n_U}, t^d_1,\ldots,t^d_{n_T}, r^d_1,\ldots,r^d_{n_R}, f^d_1,\ldots,f^d_{n_F} \rangle, \\    
\Zo_d = & \langle u^d_1,\ldots,u^d_{n_U}, t^d_1,\ldots,t^d_{n_T}, f^d_1,\ldots,f^d_{n_F} \rangle, \\  
\Bo_d = & \langle u^d_1,\ldots,u^d_{n_U}, a^{d}_{1}t^{d}_{1},\dots, a^{d}_{n_T}t^{d}_{n_T} \rangle.
\end{align*}
For each generator \(f^d_j\) in group F, the class \([f^d_j]\) generates a free summand of \(\Ho_d\). 
For each generator \(t^d_i\) in group T, the class \([t^d_i]\) generates a torsion summand satisfying \(a^d_i [t^d_i] = 0\).
Therefore, every \(d\)-homology class \(h \in \Ho_d\) admits a unique decomposition
\[
[h] =
\sum_{i=1}^{n_T} \alpha_i [t^d_i]
+
\sum_{j=1}^{n_F} \varphi_j [f^d_j]
\quad \text{with} \;
\alpha_i \in \bbd/(a^d_i), \; \varphi_j \in \bbd, \text{ and we write }
[h] =
\left(
\begin{array}{c}
\alpha_1\\
\vdots \\
\alpha_{n_T} \\
\varphi_1 \\
\vdots \\
\varphi_{n_F}
\end{array}
\right)
=
\left(
\begin{array}{c}
\alpha\\
\varphi
\end{array}
\right)
\]
to identify \([h]\) with its coordinates.
Therefore, we have the isomorphism
\[
\Ho_d \cong 
\left(\bigoplus_{i=1}^{n_T} \bbd/(a^d_i)\right)
\oplus
\bbd^{n_F} = \left\{(\alpha_{1},\dots,\alpha_{n_T}, \varphi_{1},\dots,\varphi_{n_F})' 
\middle| \; \alpha_{i} \in \bbd/(a^{d}_{i}), \varphi_{j} \in \bbd\right\}.
\]
The Betti number of \(\Ho_d\) is \(n_F\) and its torsion coefficients are \(a^{d}_{1},\ldots,a^{d}_{n_T}\) (which satisfy the divisibility condition). Thus, we recover the structure of $\Ho_d$ as presented in \Cref{thm:modules_structure}. We now describe in detail how to get the coordinates of any \(d\)-chain in both reference systems. 

Any \(d\)-chain \(c \in \Co_d\) can be decomposed uniquely as: 
\[
c = \sum_{i=1}^{N_U} \upsilon_i \cdot u^{d}_{i}
+\sum_{i=1}^{N_T} \tau_i \cdot t^{d}_{i}
+\sum_{i=1}^{N_R} \rho_i  \cdot r^{d}_{i}
+\sum_{i=1}^{N_F} \varphi_i  \cdot f^{d}_{i}, \quad \text{with} \; \upsilon_i, \tau_i, \rho_i, \varphi_i \in \bbd, \text{ and we write }
c^H = 
\left(
\begin{array}{c}
\upsilon_c \\ \tau_c \\ \rho_c \\ \varphi_c
\end{array}
\right)
\]
for some \(\upsilon_c \in \bbd^{n_U}\), \(\tau_c \in \bbd^{n_T}\), \(\rho_c \in \bbd^{n_R}\) and \(\varphi_c \in \bbd^{n_F}\). The vector of coordinates of \(c\) with respect to \(\E_d\) is 
\[
c^{\Delta}=\mdh_d c^{H} =
\left( U_d \mid T_d \mid R_d \mid F_d \right) \cdot
\left(
\begin{array}{c}
\upsilon_c \\ \tau_c \\ \rho_c \\ \varphi_c
\end{array}
\right) = U_d \upsilon_c + T_d \tau_c + R_d \rho_c + F_d \varphi_c.
\]
In this case, we write \(c \approx(\upsilon_c , \tau_c , \rho_c , \varphi_c)\) to identify \(c\) with its coordinates.

Any \(d\)-cycle \(z \in \Zo_d\) is generated by the generators in \(E^H_d\) but those in group R. Then, there exist \(\upsilon_z \in \bbd^{n_U}\), \(\tau_z \in \bbd^{n_T}\) and \(\varphi_z \in \bbd^{n_F}\) such that the vector of coordinates of \(z\) with respect to \(E^{H}_d\) has the form 
\[
z^H=
\left(
\begin{array}{c}
\upsilon_z \\ \tau_z \\ 0 \\ \varphi_z
\end{array}
\right),
\quad \text{ and therefore } \quad
z^{\Delta}=\mdh_d z^{H} = U_d \upsilon_z + T_d \tau_z + F_d \varphi_z.
\]
In this case, we write \(z \approx(\upsilon_z , \tau_z, \varphi_z)\) to identify \(z\) with its coordinates.

Any \(d\)-boundary \(b \in \Bo_d\) is generated by the generators in group U and the generators in group T multiplied by their torsion coefficients. If we define the matrix \(A_d = \operatorname{diag}(a^{d}_{1},\ldots,a^{d}_{n_T}) \in \bbd^{n_T \times n_T}\) containing the torsion coefficients of \(\Ho_d\), there exist \(\upsilon_b \in \bbd^{n_U}\) and \(\tau_b \in \bbd^{n_T}\) such that the vector of coordinates of \(b\) with respect to \(E^H_d\) has the form 
\[
b^H =
\left(
\begin{array}{c}
\upsilon_b \\ A_d \tau_b \\ 0 \\ 0
\end{array}
\right),
\quad \text{ and therefore } \quad
b^{\Delta} = \mdh_d b^H = U_d \upsilon_b + T_dA_d \tau_b.
\]
In this case, we write \(b \approx(\upsilon_b , \tau_b)\) to identify \(b\) with its coordinates.

Consider now a \(d\)-homology class \([h] = h + \Bo_d \in \Ho_d\) and a representative \(d\)-cycle \(z \in [h]\). Since \(z \in \Zo_d\), there exist vectors \(\upsilon_z,\tau_z,\varphi_z\) such that \(z \approx (\upsilon_z,\tau_z,\varphi_z)\). Because \(z \in h +\Bo_d\), there exists \(b \in \Bo_d\) such that \(z = h+b\) and there exist vectors \(\upsilon_b,\tau_b\) such that \(b \approx (\upsilon_b, \tau_b)\). Then, we may also write \(z \approx h + (\upsilon_b, \tau_b)\). 
Moreover, if \(z \approx (\upsilon_z,\tau_z,\varphi_z)\), we have that the coordinates of \([h]\) are
\[
[h] = [z] =
\left(
\begin{array}{c}
\pi^{d}_{1}(\tau_1)\\
\vdots \\
\pi^{d}_{n_T}(\tau_{n_T}) \\
\varphi_1 \\
\vdots \\
\varphi_{n_F}
\end{array}
\right)
=
\left(
\begin{array}{c}
\pi^d(\tau_z)\\
\varphi_z
\end{array}
\right),
\]
where \(\pi^{d}_{i}: \bbd \to \bbd/(a^{d}_{i})\) is the natural projection associated with
the torsion coefficient \(a^{d}_{i}\). It follows directly that these coordinates are independent of the chosen representative \(z \in [h]\).

It may seem that this exposition is overly detailed for such a basic matter. However, we have included it to clarify how the change of basis between \(E^{H}_d\) and \(\E_d\) is performed, and to introduce the notation used in next results.

Let \([h] \in \Ho_d\) be a \(d\)-homology class represented by the \(d\)-cycle \(h \in \Zo_d\). The set of \(d\)-cycles in \([h] = h +\Bo_d\) is easily described in terms of the basis \(E^{H}_d\), but \cref{item:d-chain-fuzzy} of \Cref{prop:chains_submodules_properties} shows that for each \(d\)-cycle \(z \in [h]\) the value \(\kappa_d(z) \in L\) is easier to compute having its coordinates \(z^{\Delta}\) in terms of \(\E_d\).
Because of that, we need to use the change-of-basis matrix \(\mdh_d = \left(U_d \mid T_d \mid R_d \mid F_d \right)\) described in \Cref{prop:diagonals-product-zero}.
 
Given a matrix \(M\), we denote by \(M_I\) the submatrix of \(M\) consisting of the rows indexed by \(I\).

\begin{definition}[Constraint system]
\label{def:constraint_system}
Given a \(d\)-chain \(c \in \Co_d\) and a subset 
\(I\) of the index set  \(\{1, \ldots, n_d\}\) of the oriented \(d\)-simplices in \(\E_d\), we define the following system of linear diophantine equations:
\[
S(c,I): \quad (U_{d} \mid T_{d} A_d)_I \cdot
\left(
\begin{array}{c}
\upsilon \\
\tau
\end{array}
\right)
= -c^{\Delta}_I,
\]
where \(A_d = \text{diag}(a^{d}_{1},\dots,a^{d}_{N_T})\) is  a diagonal matrix with the torsion coefficients of \(\Ho_d\), \(\upsilon \in \bbd^{n_U}\) and \(\tau \in \bbd^{n_T}\).
\end{definition}

These constraint systems are a key tool for computing the value \(\eta_d([h])\) for any \([h] \in \Ho_d\), as well as the cuts \(\eta_d^{\geq \ell}\) for \(\ell \in L\), as is shown in \Cref{thm:hdl_properties}. We now establish two technical lemmas to simplify that proof.

\begin{lemma}
\label{lem:shi_welldefined}
Let \([h] \in \Ho_d\) and let \(z \in [h]\). Then, \(S(h,I)\) is solvable if and only if \(S(z,I)\) is solvable.
\end{lemma}

\begin{proof}
Assume that \(S(h,I)\) is solvable and that 
\(\left(
\begin{array}{c}
\upsilon_* \\
\tau_*
\end{array}
\right)\) is a solution. Then
\[
(U_d \mid T_dA_d)_I
\left(
\begin{array}{c}
\upsilon_* \\
\tau_*
\end{array}
\right)
=
-h_I^\Delta .
\]

Since \(z\in[h]\), there exists \(b\in \Bo_d\) such that \(z=h+b\). Hence there exist 
\(\upsilon_b\in\bbd^{n_U}\) and \(\tau_b\in\bbd^{n_T}\) with 
\(b\approx(\upsilon_b,\tau_b)\), and \(z \approx h + (\upsilon_b,\tau_b)\). Therefore,
\[
z_I^\Delta
=
h_I^\Delta
+
(U_d\mid T_dA_d)_I
\left(
\begin{array}{c}
\upsilon_b\\
\tau_b
\end{array}
\right).
\]
Subtracting the two expressions gives
\[
(U_d\mid T_dA_d)_I
\left(
\begin{array}{c}
\upsilon_*-\upsilon_b\\
\tau_*-\tau_b
\end{array}
\right)
=
-z_I^\Delta ,
\]
so \(S(z,I)\) is solvable. The converse implication follows by reversing the argument.
\end{proof}

\begin{lemma}
\label{lem:threshold_system}
Let \(\Delta\) be a simplicial complex, let \(\mu \in \calfc(\Delta,L)\), and let \([h] \in \Ho_d\). For each \(\ell \in L\),
define the index subset
\(I(\ell) = \big\{i \in \{1, \ldots, n_d\}\mid \sigma^{d}_{i} \in \mu^{\not\geq \ell}\big\}.\)
Then, the vector 
\(\left( \begin{array}{c} \upsilon \\ \tau \end{array} \right) \in \bbd^{n_U+n_T}\)
is a solution for the system \(S(h, I(\ell))\) if and only if the \(d\)-cycle \(z \in [h]\) with \(z \approx h+(\upsilon,\tau)\) satisfies that \(\kappa_d(z)\geq \ell\).
\end{lemma}

\begin{proof}
Suppose that 
\(\left(\begin{array}{c}\upsilon \\ \tau\end{array}\right)\in \bbd^{n_U+n_T}\)
is a solution for \(S(h,I(\ell))\). This means that
\[
(U_{d} \mid T_{d} A_{d})_{I(\ell)}
\left(\begin{array}{c}\upsilon \\ \tau\end{array}\right)= -h^{\Delta}_{I(\ell)}.
\]
Consider the \(d\)-cycle \(z \in [h]\) such that \(z \approx h+(\upsilon,\tau)\). Its coordinate vector with respect to the basis \(\E_d\) is
\[
z^\Delta = h^\Delta + (U_{d} \mid T_{d}A_{d})
\left(\begin{array}{c}\upsilon \\\tau\end{array}\right).
\]
Restricting to indices in \(I(\ell)\), we obtain
\[
z^\Delta_{I(\ell)} = h^\Delta_{I(\ell)} + (U_{d} \mid T_{d} A_{d})_{I(\ell)}
\left(\begin{array}{c}\upsilon \\\tau\end{array}\right)
= h^\Delta_{I(\ell)} -h^\Delta_{I(\ell)} = 0,
\]
which implies that \(z_i = 0\) for all \(i \in I(\ell)\). By \Cref{prop:chains_submodules_properties}, we then have
\[
\kappa_d(z) = \bigwedge_{\substack{i=1,\ldots,n_d \\ z_i \neq 0}} \mu(\sigma^{d}_{i}) = \bigwedge_{\substack{i \notin I(\ell) \\ z_i \neq 0}} \mu(\sigma^{d}_{i}).
\]
By definition of \(I(\ell)\), for all \(i \notin I(\ell)\), we have \(\mu(\sigma^{d}_{i}) \geq \ell\), hence the meet is \(\kappa_d(z) \geq \ell\).

Conversely, suppose there exists a \(d\)-cycle \(z \in [h]\) such that \(\kappa_d(z) \geq \ell\). 
Recall that \(z_i\) denotes the coefficient of \(z\) with respect to the element \(\sigma^d_i \in E^\Delta_d\).
By definition of \(\kappa_d\), it is impossible to have \(z_i \neq 0\) for any \(i\) with \(\sigma^{d}_{i} \in \mu^{\not \geq \ell}\) (in that case, the meet would belong to \(L^{\not \geq \ell}\) and \(\kappa_d(z)\not\geq \ell\)). This means that \(z_i = 0\) for all \(i \in I(\ell)\). On the other hand, there exists \(b \in \Bo_d\) such that \(z=h+b\), and there exist \(\upsilon \in \bbd^{n_U}\) and \(\tau \in \bbd^{n_b}\) such that \(b \approx (\upsilon,\tau)\) and \(z \approx h+(\upsilon,\tau)\). In that case, 
\[
0=z^\Delta_{I(\ell)} = h^\Delta_{I(\ell)} + (U_{d} \mid T_{d}A_{d})_{I(\ell)}
\left(\begin{array}{c}\upsilon \\ \tau\end{array}\right).
\]
and thus 
\(\left(\begin{array}{c}\upsilon \\ \tau\end{array}\right) \in \bbd^{n_U+n_T}\)
is a solution to the system \(S(h, I(\ell))\).
\end{proof}

To solve a system of the type \(Ax=b\) with coefficients in \(\bbd\), it suffices to apply \Cref{thm:snf} and compute the two invertible matrices \(P,Q\) such that \(PAQ = D\), with \(D\) a diagonal matrix. Then, the system is equivalent to \(DQ^{-1}x = Pb\). Making the change of variables \(y = Q^{-1}x\), the new system is \(Dy = Pb\). This system is solvable if and only if each diagonal entry \(d_{i}\) of \(D\) divides the corresponding component \((Pb)_i\) for all \(i\). 

\begin{theorem}
\label{thm:hdl_properties}
Let \(\Delta\) be a simplicial complex and \(\Ho_d\) its \(d\)-homology \(\bbd\)-module. Let \(\mu \in \calfc(\Delta,L)\) be an \(L\)-fuzzy subcomplex of \(\Delta\) and \(\eta_d \in \calfm(\Ho_d,L)\) the \(L\)-fuzzy \(d\)-homology of \(\mu\). For each \(\ell \in L\),
define the crisp subset of \(\Ho_d\):
\[
\Ho_d(\ell) = \big\{[h] \in \Ho_d \mid S(h,I(\ell)) \text{ is solvable} \big\}.
\]
We claim:
\begin{enumerate}[label=(\roman*)]
\item \label{item:hdl_submodule} The set \(\Ho_d(\ell)\) is a crisp submodule of \(\Ho_d\).
\item \label{item:hdl_functor} For any \(\ell_1, \ell_2 \in L\) with \(\ell_1 \leq \ell_2\), we have \(\Ho_d(\ell_1) \supseteq \Ho_d(\ell_2)\).
\item \label{item:hdl_join} For any subset \(S \in L\), we have \(\Ho_d(\bigvee S)\subseteq \bigcap_{\ell \in S} \Ho_d(\ell)\).
\item \label{item:hdl_meet} For any subset \(S \in L\), we have \(\Ho_d(\bigwedge S)\supseteq \bigcap_{\ell \in S} \Ho_d(\ell)\).
\item \label{item:hdl_eta} For every \([h] \in \Ho_d\),
\[
\eta_d([h]) = \bigvee \Big\{\, \ell \in L(\kappa_d) \, \Big| \, [h] \in \Ho_d(\ell) \,\Big\}.
\]
\item \label{item:hdl_contain} For any \(\ell \in L\), we have \(\Ho_d(\ell) \subseteq \eta_d^{\geq \ell}\).
\item \label{item:hdl_cut} For any \(\ell \in L\),
\[
\eta_d^{\geq \ell}
= \bigcup_{\substack{S \subseteq L(\kappa_d) \\ \bigvee S \geq \ell}} \bigcap_{s \in S} \Ho_d(s).
\]
\end{enumerate}
\end{theorem}

\begin{proof}
First note that \(\Ho_d(\ell)\) is well-defined, since \Cref{lem:shi_welldefined} ensures that the solvability of \(S(h,I(\ell))\) is independent of the chosen representative of \([h]\).
\begin{enumerate}[label=(\roman*)]
\item The zero class \([0] \in \Ho_d\) belongs to \(\Ho_d(\ell)\) because \(S(0, I(\ell))\) is homogeneous and trivially solvable.
Let \([h_1], [h_2] \in \Ho_d(\ell)\). Then \(S(h_1, I(\ell))\) and \(S(h_2, I(\ell))\) have solutions 
\(\left(\begin{array}{c}\upsilon_1 \\ \tau_1\end{array}\right)\)
and 
\(\left(\begin{array}{c}\upsilon_2 \\ \tau_2\end{array}\right)\) 
respectively. Hence:
\[
(U_{d} \mid T_{d} A_d)_{I(\ell)}
\left(\begin{array}{c}\upsilon_1 \\\tau_1\end{array}\right)= -(h_1)^{\Delta}_{I(\ell)}, 
\quad
(U_{d} \mid T_{d} A_d)_{I(\ell)}
\left(\begin{array}{c}\upsilon_2 \\\tau_2\end{array}\right)= -(h_2)^{\Delta}_{I(\ell)}.
\]
Adding these equations yields:
\[
(U_{d} \mid T_{d} A_d)_{I(\ell)}
\left(\begin{array}{c}\upsilon_1 + \upsilon_2 \\\tau_1 + \tau_2\end{array}\right)= -(h_1 + h_2)^{\Delta}_{I(\ell)}.
\]
Thus, \(S(h_1 + h_2, I(\ell))\) is solvable, so \([h_1 + h_2] = [h_1] + [h_2] \in \Ho_d(\ell)\). Now, let \([h] \in \Ho_d(\ell)\) with a solution 
\(\left(\begin{array}{c}\upsilon \\ \tau\end{array}\right)\), and let \(a \in \bbd\). Then:
\[
(U_{d} \mid T_{d} A_d)_{I(\ell)}
\left(\begin{array}{c}a\upsilon \\a\tau\end{array}\right)= a \cdot
(U_{d} \mid T_{d} A_d)_{I(\ell)}
\left(\begin{array}{c}\upsilon \\\tau\end{array}\right)= -(a h)^{\Delta}_{I(\ell)}.
\]
Hence \(S(ah, I(\ell))\) is solvable and therefore \([ah] = a[h] \in \Ho_d(\ell)\).
Consequently, \(\Ho_d(\ell)\) is a submodule of \(\Ho_d\).
\item If \(\ell_1 \leq \ell_2\), then \(I(\ell_1) \subseteq I(\ell_2)\). Thus, for any \([h] \in \Ho_d\), the system \(S(h,I(\ell_2))\) contains all equations of \(S(h,I(\ell_1))\). Therefore, the solvability of \(S(h,I(\ell_2))\) implies the solvability of \(S(h,I(\ell_1))\), and hence \(\Ho_d(\ell_2) \subseteq \Ho_d(\ell_1)\).

\item From \cref{item:hdl_functor} we have \(\Ho_d(\ell) \supseteq \Ho_d \left(\bigvee S\right)\) for all \(\ell \in S\), which implies that \(\Ho_d(\bigvee S) \subseteq \bigcap_{\ell \in S} \Ho_d(\ell)\).

\item From \cref{item:hdl_functor} we have \(\Ho_d(\ell) \subseteq \Ho_d \left(\bigwedge S\right)\) for all \(\ell \in S\), which implies that \(\Ho_d(\bigwedge S) \supseteq \bigcup_{\ell \in S} \Ho_d(\ell)\).

\item\label{item:Lh-Sh} Having fixed \([h] \in \Ho_d\), define the sets:
\[
L_h = \big\{\kappa_d(z) \mid z \in [h]\big\} \subseteq L(\kappa_d), 
\qquad 
S_h = \big\{\, \ell \in L(\kappa_d) \mid [h] \in \Ho_d(\ell) \,\big\} \subseteq L(\kappa_d).
\]
By \Cref{cor:eta_value}, we know that \(\eta_d([h])=\bigvee L_h\). We now show that \(\bigvee L_h = \bigvee S_h\) by proving both inequalities.
If \(\ell \in L_h\), then there exists \(z \in [h]\) such that \(\kappa_d(z) = \ell\). By \Cref{lem:threshold_system}, \(S(h,I(\ell))\) is solvable, hence \([h] \in \Ho_d(\ell)\) and \(\ell \in S_h\). Thus \(L_h \subseteq S_h\), and consequently \(\bigvee L_h \leq \bigvee S_h\). Conversely, if \(\ell \in S_h\), then \([h] \in \Ho_d(\ell)\) and \(S(h,I(\ell))\) is solvable. By \Cref{lem:threshold_system}, there exists \(z \in [h]\) such that \(\kappa_d(z) \geq \ell\), with \(\kappa_d(z) \in L_h\). Therefore, each element of \(S_h\) is bounded by an element of \(L_h\), implying \(\bigvee S_h \leq \bigvee L_h\). Hence \(\eta_d([h])=\bigvee L_h = \bigvee S_h\).
\item Let \([h] \in \Ho_d(\ell)\).
Then $\ell \in S_h$; where $S_h$ is defined as in \cref{item:Lh-Sh}. 
Further, by \cref{item:hdl_eta}, we have \(\eta_d([h])=\bigvee S_{h}\). Hence \(\ell \leq \eta_d([h])\), which implies that \([h] \in \eta_d^{\geq \ell}\). Therefore \(\Ho_d(\ell) \subseteq \eta_d^{\geq \ell}\).
\item We prove both inclusions. First assume that \([h] \in \eta_d^{\geq \ell}\).
By \cref{item:hdl_eta}, we have \(\eta_d([h])=\bigvee S_h \geq \ell\). Since \([h] \in \Ho_d(s)\) for every \(s \in S_h\), it follows that \([h] \in \bigcap_{s \in S_h} \Ho_d(s)\). Because \(\bigvee S_h \geq \ell\), we obtain
\[
[h] \in 
\bigcup_{\substack{S \subseteq L(\kappa_d) \\ \bigvee S \geq \ell}}
\bigcap_{s \in S} \Ho_d(s).
\]
Conversely, assume that \([h]\) belongs to that set.
Then there exists a subset \(S \subseteq L(\kappa_d)\) such that \([h] \in \Ho_d(s)\) for all \(s \in S\) and \(\bigvee S \geq \ell\).
In particular, $S \subseteq S_h$.
By \cref{item:hdl_eta},
\(\eta_d([h])= \bigvee S_h\geq \bigvee S\geq \ell.\)
Therefore \([h] \in \eta_d^{\geq \ell}\), which completes the proof. \qedhere
\end{enumerate}
\end{proof}

Let us discuss why this result is relevant. We already know from \Cref{prop:submodules_properties} that \(\cut(\eta_d):L \to \sub(\Ho_d)\) is a contravariant functor given by \(\ell \mapsto \eta_d^{\geq \ell}\). Now, in \Cref{thm:hdl_properties} we introduce a second contravariant functor \(\solv(\eta_d): L \to \sub(\Ho_d)\) given by \(\ell \mapsto \Ho_d(\ell)\), whose definition is based on the solvability of certain linear systems. By \cref{item:hdl_contain}, the functor \(\solv(\eta_d)\) is pointwise contained in \(\cut(\eta_d)\), that is, \(\Ho_d(\ell) \subseteq \eta_d^{\geq \ell}\) for all \(\ell \in L\). The reverse inclusion is not true in general, as shown in \Cref{rem:counterexample1}.
The new functor \(\solv(\eta_d)\) does not preserve joins though. Indeed, while \Cref{prop:subsetsproperties} ensures that \(\eta_d^{\geq \bigvee S}= \bigcap_{\ell \in S}\eta_d^{\geq \ell}\), it is not generally true that \(\Ho_d(\bigvee S) = \bigcap_{\ell \in S} \Ho_d(\ell)\), as shown in \Cref{rem:counterexample2}.

However, the family of submodules \(\Ho_d(\ell)\) encodes all the information needed to recover \(\eta_d\). More precisely, \cref{item:hdl_eta} provides an alternative characterization of the value \(\eta_d([h])\) in terms of the submodules \(\Ho_d(\ell)\). This characterization is particularly useful from a computational point of view. Indeed, the original definition of \(\eta_d([h])\) requires considering all \(d\)-cycles in the class \([h]\), while the new formulation only involves the values in \(L(\kappa_d)\), which is always a finite set (since we work with finite simplicial complexes). Furthermore, \cref{item:hdl_cut} shows that the cuts \(\eta_d^{\geq \ell}\) can be reconstructed from the family of submodules \(\Ho_d(\ell)\) by means of unions and intersections. Consequently, the computation of the \(L\)-fuzzy homology reduces to the computation of the submodules \(\Ho_d(\ell)\).

\begin{remark}
The formula in \cref{item:hdl_cut} of \Cref{thm:hdl_properties}
\[
\eta_d^{\geq \ell}
= \bigcup_{\substack{S \subseteq L(\kappa_d) \\ \bigvee S \geq \ell}}
\bigcap_{s \in S} \Ho_d(s)
\]
may appear complicated at first sight, but it simplifies considerably in common situations. 
For instance, suppose that the CDL \((L,\leq)\) is totally ordered. Let \(S \subseteq L(\kappa_d)\) such that \(\bigvee S \geq \ell\). 
Since \(L(\kappa_d)\) is finite and totally ordered, the join \(\bigvee S\) is simply the maximum element of \(S\), and it belongs to \(S\).
By \cref{item:hdl_functor} of \Cref{thm:hdl_properties}, we have \(\Ho_d(s) \supseteq \Ho_d(\bigvee S)\) for all \(s \in S\). 
Therefore, the formula becomes
\[
\eta_d^{\geq \ell}
= \bigcup_{\substack{S \subseteq L(\kappa_d) \\ \bigvee S \geq \ell}}
\Ho_d \left(\bigvee S\right).
\]
However, if \(S\) has maximum \(s^*=\bigvee S\), the singleton \(\{s^*\}\) also satisfies 
\(\bigvee\{s^*\}=s^*\). Hence subsets with more than one element are redundant in the union, and the expression reduces to
\[
\eta_d^{\geq \ell}
= \bigcup_{\substack{s \in L(\kappa_d) \\ s \geq \ell}} \Ho_d(s).
\]
Since the set \(\{s \in L(\kappa_d)\mid s \geq \ell\}\) is finite and totally ordered, it has a minimum element \(t\). Again by \cref{item:hdl_functor} of \Cref{thm:hdl_properties}, we have \(\Ho_d(t) \supseteq \Ho_d(s)\) for every \(s \geq \ell\), and therefore
\[
\eta_d^{\geq \ell}=\Ho_d(t),
\quad \text{ being }
t=\min\{s \in L(\kappa_d)\mid s \geq \ell\}.
\]
\end{remark}

With these results in place, we now proceed to compute explicitly the submodules \(\Ho_d(\ell)\) for each \(\ell \in L(\kappa_d)\). Recall that we have the isomorphism \(\Ho_d \cong 
\{(\alpha_{1},\dots,\alpha_{n_T},\varphi_{1},\dots,\varphi_{n_F})' 
\mid \alpha_{i} \in \bbd/(a^{d}_{i}),\ \varphi_{j} \in \bbd\}\).

\begin{proposition}
\label{prop:hdl_projection}
Let \(\Delta\) be a simplicial complex, let \(\Ho_d\) denote its \(d\)-homology, let \(\mu \in \calfc(\Delta,L)\) and let \(\ell \in L\). Let \(G_d=(U_d \mid T_d \mid F_d)\) be the submatrix of \(M_d^{\Delta,H}\) without the block \(R_d\) and denote \(G_{d,I(\ell)} = (U_d \mid T_d \mid F_d)_{I(\ell)}\). Then,
\[
\Ho_d(\ell)= 
\left\{ 
[h] \in \Ho_d \; \middle | \;
h \approx (\upsilon,\tau,\varphi) \text{  and  }
\left(\begin{array}{c}\upsilon \\ \tau \\ \varphi \end{array}\right) \in \ker G_{d,I(\ell)}
\right\}.
\]
\end{proposition}

\begin{proof}
We prove both inclusions, denoting \(G_d(\ell)\) to the set in the right side. In the first place, let \([h] \in G_d(\ell)\). Then, there exist \(\upsilon \in \bbd^{n_U}\), \(\tau \in \bbd^{n_T}\) and \(\varphi \in \bbd^{n_F}\) such that \(h \approx (\upsilon,\tau,\varphi)\) and
\[
h^{\Delta}_{I(\ell)}= 
\left(U_{d} \mid T_{d} \mid F_{d}\right)_{I(\ell)}
\left(\begin{array}{c}\upsilon\\ \tau\\ \varphi\end{array}\right)
= G_{d,I(\ell)}
\left(\begin{array}{c}\upsilon\\ \tau\\ \varphi\end{array}\right) = 0.
\]
Since \(h^{\Delta}_{I(\ell)} =0\), the system \(S(h,I(\ell))\) is homogeneous and trivially solvable, and therefore \([h] \in \Ho_d(\ell)\).

Conversely, consider a \(d\)-homology class \([h] \in \Ho_d(\ell)\). Since \(h \in \Zo_d\), there exist \(\upsilon \in \bbd^{n_U}\), \(\tau \in \bbd^{n_T}\) and \(\varphi \in \bbd^{n_F}\) such that \(h \approx (\upsilon,\tau,\varphi)\). Then, we have
\[
h^{\Delta}=
\left(U_d \mid T_d \mid F_d\right)
\left(\begin{array}{c}\upsilon\\ \tau\\ \varphi\end{array}\right)
=G_d
\left(\begin{array}{c}\upsilon\\ \tau\\ \varphi\end{array}\right)
\quad \text{ and } \quad
[h] = 
\left(\begin{array}{c}\pi^{d}(\tau) \\ \varphi\end{array}\right).
\]
Since \([h] \in \Ho_d(\ell)\), the system \(S(h,I(\ell))\) is solvable.
Thus there exist \(\upsilon_h \in \bbd^{n_U}\) and \(\tau_h \in \bbd^{n_T}\) such that
\[
\left(U_{d} \mid T_{d} A_d\right)_{I(\ell)}
\left(\begin{array}{c}\upsilon_h\\ \tau_h\end{array}\right)
=- h^{\Delta}_{I(\ell)}=
-\left(U_{d} \mid T_{d} \mid F_{d}\right)_{I(\ell)}
\left(\begin{array}{c}\upsilon\\ \tau\\ \varphi\end{array}\right).
\]
Rewriting, we have
\[
\left(U_{d} \mid T_{d} \mid F_{d}\right)_{I(\ell)}
\left(\begin{array}{c}\upsilon + \upsilon_h\\\tau + A_d \tau_h\\\varphi\end{array}\right)
=0,
\quad \text{ and } \quad
\left(\begin{array}{c}\upsilon + \upsilon_h\\\tau + A_d \tau_h\\\varphi\end{array}\right) 
\in \ker G_{d,I(\ell)}.
\]
Let \(z \in \Zo_d\) be the \(d\)-cycle such that \(z \approx (\upsilon+\upsilon_h, \tau + A_d\tau_h,\varphi)\). The \(d\)-homology class \([z]\) belongs to \(G_d(\ell)\). Moreover, for \(i=1,\dots,n_T\) we have that \(\pi^d_i(\tau_i + a^d_i\tau_{h,i})=\pi^d_i(\tau_i)\). Then,
\[
[z] = 
\left(\begin{array}{c}\pi^d(\tau + A_d \tau_h) \\ \varphi \end{array}\right) = 
\left(\begin{array}{c}\pi^d(\tau) \\ \varphi \end{array}\right) = [h].
\]
This implies that \([h]=[z]\) and \([h]\) also belongs to \(G_d(\ell)\).
\end{proof}

In summary, the computation of \(\Ho_d(\ell)\) consists of computing the kernel of \(G_{d,I(\ell)}\), which is a subset of \(\Zo_d\), and projecting the solutions to \(\Ho_d\). The kernel of the matrices \(G_{d,I(\ell)}\) can also be computed via the Smith normal form, since that is equivalent to solving the homogeneous system \(G_{d,I(\ell)} \xi = 0\). Alternatively, other triangularization-based techniques, such as the Hermite Normal Form \cite[Chapter~5]{schrijver1998theory}, may be employed.

\section{Example of computation}
\label{sec:example}

\begin{figure}
\centering
\begin{subfigure}[b]{0.45\textwidth}
\centering
\begin{tikzpicture}[scale=2, every node/.style={font=\small}]
\coordinate (v0) at (0,0);
\coordinate (v1) at (2,0);
\coordinate (v2) at (2,2);
\coordinate (v3) at (0,2);
\coordinate (v4) at (0.5,0.5);

\draw[line width=0.8pt] (v0)--(v1); 
\draw[line width=0.8pt] (v1)--(v2);
\draw[line width=0.8pt] (v2)--(v3); 
\draw[line width=0.8pt] (v0)--(v3); 
\draw[line width=0.8pt] (v1)--(v3); 

\filldraw[fill=gray!20, draw=black, line width=0.6pt] (v1)--(v2)--(v3)--cycle;

\fill[myred] (v0) circle (2pt);
\draw (v0) circle (2pt);
\fill[myred] (v1) circle (2pt);
\draw (v1) circle (2pt);
\fill[myblue] (v2) circle (2pt);
\draw (v2) circle (2pt);
\fill[myred] (v3) circle (2pt);
\draw (v3) circle (2pt);
\fill[myblue] (v4) circle (2pt);
\draw (v4) circle (2pt);

\node at (v0) [font=\small, anchor=north east] {\(v_0\)};
\node at (v1) [font=\small, anchor=north west] {\(v_1\)};
\node at (v2) [font=\small, anchor=north west] {\(v_2\)};
\node at (v3) [font=\small, anchor=north east] {\(v_3\)};
\node at (v4) [font=\small, anchor=south west] {\(v_4\)};
\end{tikzpicture}
\subcaption{Simplicial complex \(\Delta\) on a bi-chromatic dataset.}
\label{fig:delta}
\end{subfigure}
\hfill
\begin{subfigure}[b]{0.45\textwidth}
\centering
\begin{tikzpicture}[scale=2, every node/.style={font=\small}]
\coordinate (v0) at (0,0);
\coordinate (v1) at (2,0);
\coordinate (v2) at (2,2);
\coordinate (v3) at (0,2);
\coordinate (v4) at (0.5,0.5);

\filldraw[fill=gray!20, draw=black, line width=0.6pt] (v1)--(v2)--(v3)--cycle;
\node at (1.5,1.5) [font=\small] {\(x \wedge y\)};

\draw[line width=0.8pt] (v0)--(v1) node[midway, below] {\(x\)};
\draw[line width=0.8pt] (v1)--(v2) node[midway, right] {\(x \wedge y\)};
\draw[line width=0.8pt] (v2)--(v3) node[midway, above] {\(x\wedge y\)};
\draw[line width=0.8pt] (v0)--(v3) node[midway, left] {\(x\)};
\draw[line width=0.8pt] (v1)--(v3) node[midway, sloped, above] {\(x\)};

\fill (v0) circle (2pt);
\draw (v0) circle (2pt);
\fill (v1) circle (2pt);
\draw (v1) circle (2pt);
\fill (v2) circle (2pt);
\draw (v2) circle (2pt);
\fill (v3) circle (2pt);
\draw (v3) circle (2pt);
\fill (v4) circle (2pt);
\draw (v4) circle (2pt);

\node at (v0) [anchor=north east] {\(x\)};
\node at (v1) [anchor=north west] {\(x\)};
\node at (v2) [anchor=north west] {\(y\)};
\node at (v3) [anchor=north east] {\(x\)};
\node at (v4) [anchor=south west] {\(y\)};
\end{tikzpicture}
\subcaption{\(L\)-fuzzy subcomplex of \(\Delta\) defined by the two colors.}
\label{fig:mu}
\end{subfigure}
\caption{Example of a simplicial complex \(\Delta\) on a bi-chromatic dataset and an \(L\)-fuzzy subcomplex \(\mu \in \calfc(\Delta,L)\), being \((L\leq)=\fdl(x,y)\). Red points are assigned value \(x\) and blue points are assigned value \(y\).}
\label{fig:delta_mu}
\end{figure}
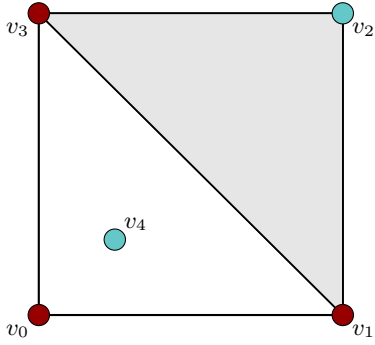
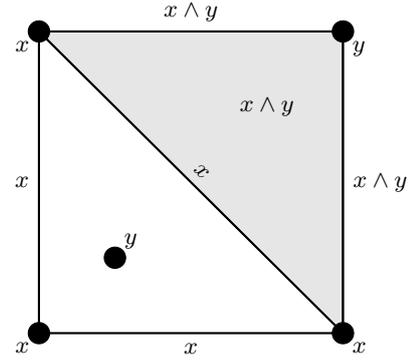

In this section we apply the theoretical results from \Cref{sec:computation} on a toy example. Consider the bi-chromatic dataset depicted in \Cref{fig:delta}, which has three red points and two blue points. We build on top of it a simplicial complex, which we denote by \(\Delta\), consisting of five \(0\)-simplices, five \(1\)-simplices, and one \(2\)-simplex. Explicitly,
\[
\Delta = 
\left\{
\overbrace{\langle v_0\rangle,\; \langle v_1\rangle,\; \langle v_2\rangle,\; \langle v_3\rangle,\; \langle v_4\rangle}^{\Delta_0},\;
\overbrace{\langle v_0,v_1\rangle,\; \langle v_0,v_3\rangle,\; \langle v_1,v_2\rangle,\;\langle v_1,v_3\rangle,\; \langle v_2,v_3\rangle}^{\Delta_1},\;
\overbrace{\langle v_1,v_2,v_3\rangle}^{\Delta_2}
\right\}.
\]
Consider the free distributive lattice \((L,\leq)=\fdl(x,y)\). 
We define a bi-chromatic \(L\)-fuzzy subcomplex \(\mu \in \calfc(\Delta,L)\) as in \Cref{exa:chroTDA}. Each red vertex is assigned the value \(x\), each blue vertex is assigned the value \(y\), and the value of any other simplex depends on the colors involved. The concrete values of \(\mu\) can be seen in \Cref{fig:mu}.

Throughout the example we take \(\bbz\) as the coefficient ring for homology computations. Since \(\mu^*=\Delta\) and \(0\) is meet-prime in \(\fdl(x,y)\), it follows that the associated fuzzy homology maps \(\eta_0 \in \calfm(\Ho_0,L)\) and \(\eta_1 \in \calfm(\Ho_1,L)\) are well defined. We carry out in detail the following tasks:
\begin{enumerate}
\item We compute the homology \(\bbz\)-modules (abelian groups) \(\Ho_0\) and \(\Ho_1\) applying \Cref{prop:diagonals-product-zero}.
\item We compute the value \(\eta_1([h])\) for a \(1\)-cycle \(h \in \Zo_1\) applying \cref{item:hdl_eta} of \Cref{thm:hdl_properties}.
\item We compute the cuts \(\eta_0^{\geq \ell}\) for each \(\ell \in L\) applying \Cref{prop:hdl_projection} and \cref{item:hdl_cut} of \Cref{thm:hdl_properties}.
\end{enumerate}

\paragraph{1. Computation of the homology groups}
If we order the vertex set lexicographically (\(v_0 < v_1 < v_2 < v_3 < v_4\)), the positively oriented simplices of \(\Delta\) are:
\[
\begin{array}{lllll}
E^{\Delta}_0 = \big\{\,\sigma^0_1=[v_0], & \sigma^0_2=[v_1], & \sigma^0_3=[v_2], & \sigma^0_4=[v_3], & \sigma^0_5=[v_4]\,\big\}\,, \\[0.6em]
E^{\Delta}_1 = \big\{\,\sigma^1_{1}=[v_0,v_1], & \sigma^1_{2}=[v_0,v_3], & \sigma^1_{3}=[v_1,v_2], & \sigma^1_{4}=[v_1,v_3], & \sigma^1_{5}=[v_2,v_3]\,\big\}\,, \\[0.6em]
E^{\Delta}_2 = \big\{\,\sigma^{2}_{1}=[v_1,v_2,v_3]\,\big\}\,. & & & &
\end{array}
\]
The sets \(E^{\Delta}_{0}\), \(E^{\Delta}_{1}\) and \(E^{\Delta}_{2}\) serve as bases for the \(d\)-chain \(\bbz\)-modules 
\(\Co_0\), \(\Co_1\) and \(\Co_2\) respectively. Given these bases, the boundary operators \(\partial_0: \Co_0 \to \Co_{-1}=\{0\}\), \(\partial_1: \Co_1 \to \Co_0\) and \(\partial_2: \Co_2 \to \Co_1\) are represented by the matrices:
\[
M_0=
\left(\begin{array}{ccccc} 0 & 0 & 0 & 0 & 0 \end{array}\right), \; \;
M_1=
\left(
\begin{array}{rrrrr}
-1 & -1 &  0 &  0 &  0 \\
 1 &  0 & -1 & -1 &  0 \\
 0 &  0 &  1 &  0 & -1 \\
 0 &  1 &  0 &  1 &  1 \\
 0 &  0 &  0 &  0 &  0 
\end{array}
\right)
\; \; \text{ and } M_2=
\left(\begin{array}{r} 0 \\ 0 \\ 1 \\-1 \\ 1 \end{array}\right),
\]
which satisfy \(M_0 M_1=0\) and \(M_1M_2=0\). After completing the reduction process detailed within \Cref{prop:diagonals-product-zero}, we obtain the following diagonal matrices:
\[
D_0=
\left(
\begin{array}{ccccc}
0 & 0 & 0 & 0 & 0 
\end{array}
\right), \; \;
D_1=
\left(
\begin{array}{rrrrr}
 0 &  1 &  0 &  0 &  0 \\
 0 &  0 &  1 &  0 &  0 \\
 0 &  0 &  0 &  1 &  0 \\
 0 &  0 &  0 &  0 &  0 \\
 0 &  0 &  0 &  0 &  0 
\end{array}
\right)
\; \; \text{ and }
D_2=
\left(\begin{array}{r} 1 \\ 0 \\ 0 \\ 0 \\ 0\end{array}\right),
\]
being the change-of-basis matrices:
\[
\mdh_0=
\left(
\begin{array}{rrrrr}
 1 &  0 &  0 &  0 &  0 \\ 
 0 &  1 &  0 &  0 &  0 \\ 
 0 &  0 &  1 &  0 &  0 \\ 
-1 & -1 & -1 &  1 &  0 \\ 
 0 &  0 &  0 &  0 &  1
\end{array}
\right),\; \;
\mdh_1=
\left(
\begin{array}{rrrrr}
 0 &  0 &  1 &  0 &  1 \\ 
 0 & -1 & -1 &  0 & -1 \\ 
 1 &  0 &  0 &  0 &  0 \\ 
-1 &  0 &  0 &  0 &  1 \\ 
 1 &  0 &  0 & -1 &  0
\end{array}
\right)
\; \; \text{ and }
\mdh_2=
\left(1\right)
.
\]
The matrices \(D_0\), \(D_1\) and \(D_2\) represent \(\partial_0\), \(\partial_1\) and \(\partial_2\) respect to the new bases \(E^{H}_0\),\(E^{H}_1\) and \(E^{H}_2\). We now examine \(D_0\), \(D_1\) and \(D_2\) to find the different groups in \(E^{H}_0\) and \(E^{H}_1\).

In dimension 0, the group U is formed by \(u^0_1=\sigma^0_1-\sigma^0_{4}\), \(u_{2}=\sigma^0_{2}-\sigma^0_{4}\) and \(u^0_{3}=\sigma^0_{3}-\sigma^0_{4}\) because the three first columns of \(D_0\) are zero and the three first rows of \(D_1\) have a unit. The group T is empty because \(D_1\) only contains zeros and units. The group R is also empty because there are not non-zero columns in \(D_0\). Finally, the group F is formed by \(f^0_1=\sigma^0_4\) and \(f^0_2=\sigma^{0}_{5}\) because the last two columns of \(D_0\) are zero and the last two rows of \(D_1\) are zero too. The change-of-basis matrix \(M^{\Delta, H}_0\) would be then divided into the blocks:
\[
U_0=
\left(
\begin{array}{rrr}
 1 &  0 &  0 \\
 0 &  1 &  0 \\
 0 &  0 &  1 \\
-1 & -1 & -1 \\
 0 &  0 &  0
\end{array}
\right),
\quad T_0=
\left(\begin{array}{r}\;\\\;\\\;\\\;\\\;\end{array}\right),
\quad R_0=
\left(\begin{array}{r}\;\\\;\\\;\\\;\\\;\end{array}\right)
\text{  and  } 
\;F_0=
\left(
\begin{array}{rr}
 0 &  0 \\
 0 &  0 \\
 0 &  0 \\
 1 &  0 \\
 0 &  1
\end{array}
\right).
\]
Then, \(\Ho_0\) is generated by the free \(0\)-homology classes \([f^0_1]=[\sigma^0_4]\) and \([f^0_{2}]=[\sigma^0_{5}]\) (which represent the connected components of \(v_3\) and \(v_4\) respectively) and: 
\[
\Ho_0 = \langle [f^0_1], [f^0_2] \rangle \cong \left\{
\left(\begin{array}{c}\varphi_1 \\ \varphi_2 \end{array}\right)
\; \middle | \; \varphi_i \in \bbz
\right\} = \bbz^2\,.
\]

In dimension 1, the group U is formed by \(u^1_{1}=\sigma^1_{3}-\sigma^1_{4}+\sigma^1_{5}\) because the first column of \(D_1\) is zero and the first row of \(D_2\) has a unit. The group T is empty because \(D_2\) only contains zeros and units. The group R is formed by \(r^1_{1}=-\sigma^1_{2}\), \(r^1_{2}=\sigma^1_{1} - \sigma^1_{2}\) and \(r^1_{3}=-\sigma^1_{5}\) because the second, third and fourth columns of \(D_1\) are non-zero. Finally, the group F is formed by \(f^1_{1}=\sigma^1_{1}-\sigma^1_{2}+\sigma^1_{4}\) because the fifth column of \(D_1\) is null and the fifth row of \(D_2\) is zero too. The change-of-basis matrix \(M^{\Delta, H}_1\) would be then divided into the blocks:
\[
U_1=
\left(\begin{array}{r} 0 \\ 0 \\ 1 \\-1 \\ 1\end{array}\right),
\quad T_1=
\left(\begin{array}{r}\;\\\;\\\;\\\;\\\;\end{array}\right),
\quad R_1 = 
\left(
\begin{array}{rrr}
 0 &  1 &  0 \\
-1 & -1 &  0 \\
 0 &  0 &  0 \\
 0 &  0 &  0 \\
 0 &  0 &  -1 
\end{array}
\right)
\text{  and  } \;
F_1=
\left(\begin{array}{r} 1 \\-1 \\ 0 \\ 1 \\ 0\end{array}\right).
\]
Then, \(\Ho_1\) is generated by the free \(1\)-homology class \([f^{1}_{1}]=[\sigma^1_{1}-\sigma^1_{2}+\sigma^1_{4}]\) (which represents the loop that goes through the vertices \(v_0\), \(v_1\), \(v_3\) and \(v_0\) again) and
\[
\Ho_1 = \langle [f^1_1] \rangle \cong \{\varphi_1 \mid \varphi_1 \in \bbz\} = \bbz.
\]

\paragraph{2. L-fuzzy value of a homology class} Now, it is our turn to compute \(\eta_1([f])\), where 
\[
f = f^1_1 = \sigma^{1}_{1}-\sigma^{1}_{2}+\sigma^{1}_{4} = [v_0,v_1] - [v_0,v_3] + [v_1,v_3] =  [v_0,v_1]+[v_1,v_3]+[v_3,v_0] \in \Zo_1
\]
is the loop that goes through \(v_0\), \(v_1\), \(v_3\) and \(v_0\) again. Its homology class \([f]\) generates the \(\bbz\)-module \(\Ho_1\). 

In the first place, we interpret the value of \(\kappa_1(c)\) for any \(1\)-chain \(c \in \Co_1\), recalling \Cref{prop:chains_submodules_properties} and noting that the vertices in a bi-chromatic subcomplex take only the values \(x\) (for red points) and \(y\) (for blue points). Then:
\begin{itemize}
\item[--] If \(\kappa_1(c)=x\), all \(1\)-simplices appearing in \(c\) have value \(x\). In other words, \(c\) is a red \(1\)-chain.
\item[--] If \(\kappa_1(c)=y\), then \(c\) is a blue \(1\)-chain.
\item[--] If \(\kappa_1(c)=x \wedge y\), then \(c\) is a red-and-blue \(1\)-chain.
\item[--] If \(\kappa_1(c)=1\), then necessarily \(c=0\).
\item[--] No \(1\)-chain \(c\) satisfies \(\kappa_1(c)=x \vee y\) or \(\kappa_1(c)=0\).
\end{itemize}
The same conclusions hold analogously for any \(\kappa_d\). We now interpret the value \(\eta_1([h])\) for any \(1\)-homology class \([h]\in \Ho_1\). By \Cref{def:eta}, we have \(\eta_1([h])=\bigvee \{\kappa_1(c) \mid c \in [h]\}\). Thus:
\begin{itemize}
\item[--] If \(\eta_1([h])=x \wedge y\), then \([h]\) can be represented by a red-and-blue \(1\)-chain.
\item[--] If \(\eta_1([h])=x\), then \([h]\) can be represented by a red \(1\)-chain.
\item[--] If \(\eta_1([h])=y\), then \([h]\) can be represented by a blue \(1\)-chain.
\item[--] If \(\eta_1([h])=x \vee y\), then \([h]\) can be represented independently by a red \(1\)-chain and a blue \(1\)-chain.
\item[--] If \(\eta_1([h])=1\), then \([h]=[0]\).
\item[--] No \(1\)-homology class \([h]\) satisfies \(\eta_1([h])=0\).
\end{itemize}
The same conclusions hold for any \(\eta_d\). We are now ready to compute \(\eta_1([f])\) by applying \cref{item:hdl_eta} of \Cref{thm:hdl_properties}, and to interpret the resulting value. Note that 
\[
L(\delta_1) \setminus \{0\} = \{x\wedge y, x\} \quad \text{and} \quad L(\kappa_1) = \left\{\bigwedge S \middle | S \subseteq L(\delta_1)\setminus \{0\}\right\}=  \{x\wedge y, x, 1\}.
\]
For each \(\ell \in L(\kappa_1)\), we check whether the system \(S(f,I(\ell))\) has any solution or not. As discussed before, the block \(U_1\) only contains one column referring to the \(1\)-boundary \(u^{1}_{1} = \sigma^{1}_{3}- \sigma^{1}_{4}+ \sigma^{1}_{5}\) and the block \(T_1\) is empty. Then, the complete system \(S(f,\{1,\dots,n_1\})\) is:
\[
\left(\begin{array}{r} 0 \\ 0 \\ 1 \\-1 \\ 1\end{array}\right)
\cdot \left(\upsilon_1\right)=
\left(\begin{array}{r}-1 \\ 1 \\ 0 \\-1 \\ 0\end{array}\right).
\]
Applying \Cref{lem:threshold_system}, the existence of a solution \(\upsilon_1 \in \bbz\) to the system \(S(f,I(\ell))\) ensures that the \(1\)-cycle \(z = f + \upsilon_1 \, u^{1}_{1} \in [f]\) satisfies \(\kappa_1(z)\geq \ell\).
The index set \(I(\ell)\) contains \(i\) if and only if \(\mu(\sigma^{1}_{i})\not \geq \ell\). Considering that
\[
\mu(\sigma^1_1)=x, \quad \mu(\sigma^1_2)=x, \quad \mu(\sigma^1_3)=x \wedge y, \quad \mu(\sigma^1_4)=x, \quad \mu(\sigma^1_5)=x \wedge y,
\]
we have: 
\begin{itemize}
\item \(I(x \wedge y) = \emptyset\) because \(\nu(\sigma^{1}_{i})\geq x \wedge y\) for all \(i=1,2,3,4,5\).
\item \(I(x) = \{3,5\}\) because \(\nu(\sigma^{1}_{i}) = x \wedge y \not\geq x \) for \(i=3,5\) and \(\nu(\sigma^{1}_{i})=x\) for \(i=1,2,4\).
\item \(I(1) = \{1,2,3,4,5\}\) because \(\nu(\sigma^{1}_{i}) \not\geq 1 \) for \(i=1,2,3,4,5\).
\end{itemize}
Then, for each \(\ell \in L(\kappa_1)\):
\begin{itemize}
\item \(S(f,I(x \wedge y))\) is an empty system that is trivially solved by any \(\upsilon_1 \in \bbz\).
\item \(S(f,I(x))\) is a system with only the third and fifth rows of \((U_1|T_1A_1)\) and \(-f^{\Delta}\):
\[
\left(\begin{array}{r}1 \\1\end{array}\right)
\cdot \left(\upsilon_1\right) =
\left(\begin{array}{r}0 \\0\end{array}\right),
\]
whose only solution is \(\upsilon_1=0\).
\item \(S(f,I(1))\) is the complete system 
\[
\left(\begin{array}{r} 0 \\ 0 \\ 1 \\-1 \\  1\end{array}\right)
\cdot \left(\upsilon_1\right)=
\left(\begin{array}{r}-1 \\ 1 \\ 0 \\-1 \\ 0\end{array}\right),
\]
which is unsolvable because the two first equations are unsolvable too.
\end{itemize}
Then, by \cref{item:hdl_eta} of \Cref{thm:hdl_properties}, we obtain \(\eta_1([f]) = \bigvee \{x \wedge y, x\} = x\), meaning that \([f]\) can be represented by a red chain. In this case, that red chain is \(f\), which goes through the red vertices \(v_0\), \(v_1\), and \(v_3\).

\paragraph{3. Describing the cuts family}
Now, we describe the \(L\)-fuzzy submodule \(\eta_0 \in \calfm(\Ho_0,L)\) via its cuts \(\eta_0^{\geq \ell}\). 

In the first place, we interpret the \(\Ho_0(\ell)\) submodules. If \([h] \in \Ho_0(\ell)\), then the system \(S(h,I(\ell))\) is solvable and, by \Cref{lem:threshold_system}, that implies that there exists a \(0\)-cycle \(z \in [h]\) such that \(\kappa_0(z) \geq \ell\). Thus:
\begin{itemize}
\item[--] If \([h] \in \Ho_0(1)=\Ho_0(x \vee y)\), then \([h]\) contains the zero \(0\)-chain and \([h]=[0]\).
\item[--] If \([h] \in \Ho_0(x)\), then either \([h]=[0]\) or \([h]\) can be represented by a red \(0\)-chain.
\item[--] If \([h] \in \Ho_0(y)\), then either \([h]=[0]\) or \([h]\) can be represented by a blue \(0\)-chain.
\item[--] Since \(\kappa_0(c) \geq x \wedge y\) for any \(c \in \Co_0\), then \(\Ho_0(x \wedge y)= \Ho_0\).
\end{itemize}
We are now ready to compute the \(\Ho_0(\ell)\) submodules for all \(\ell \in L(\kappa_0)\). To compute them, note that 
\[
L(\delta_0) \setminus \{0\} = \{x,y\} \quad \text{and} \quad L(\kappa_0) = \left\{\bigwedge S \middle | S \subseteq L(\delta_0)\setminus \{0\}\right\}=  \{x\wedge y, x, y, 1\}.
\]
By \Cref{prop:hdl_projection}, each \(\Ho_0(\ell)\) is computed by finding the kernel of the matrix \(G_{0,I(\ell)}\) and projecting it to \(\Ho_0\). As discussed before, the block \(R_0\) is empty. That means that the matrix \(G_0=(U_0|T_0|F_0)\) coincides exactly with \(\mdh_0\). Then, the complete system \(G_0 \xi = 0\) is
\[
\left(
\begin{array}{rrrrr}
 1 &  0 &  0 &  0 &  0 \\ 
 0 &  1 &  0 &  0 &  0 \\ 
 0 &  0 &  1 &  0 &  0 \\ 
-1 & -1 & -1 &  1 &  0 \\ 
 0 &  0 &  0 &  0 &  1
\end{array}
\right) \cdot
\left(\begin{array}{r} \upsilon_1  \\  \upsilon_2  \\  \upsilon_3  \\  \varphi_1  \\  \varphi_2 \end{array}\right) =
\left(\begin{array}{r}0 \\ 0 \\ 0 \\ 0  \\ 0 \end{array}\right).
\]
The index set \(I(\ell)\) contains \(i\) if and only if \(\mu(\sigma^{0}_{i})\not \geq \ell\). Considering that
\[
\mu(\sigma^0_1)=x, \quad \mu(\sigma^0_2)=x, \quad \mu(\sigma^0_3)=y, \quad \mu(\sigma^0_4)=x, \quad \mu(\sigma^0_5)=y,
\]
we have:
\begin{itemize}
\item \(I(x \wedge y) = \emptyset\) because \(\mu(\sigma^{0}_{i})\geq x \wedge y\) for all \(i=1,2,3,4,5\).
\item \(I(x) = \{3,5\}\) because \(\mu(\sigma^{0}_{i}) =y \not \geq x\) for \(i=3,5\) and \(\mu(\sigma^{0}_{i})\geq x\) for \(i=1,2,4\).
\item \(I(y) = \{1,2,4\}\) because \(\mu(\sigma^{0}_{i}) =x \not \geq y\) for \(i=1,2,4\) and \(\mu(\sigma^{0}_{i})\geq y\) for \(i=3,5\).
\item \(I(1) = \{1,2,3,4,5\}\) because \(\mu(\sigma^{0}_{i})\not \geq 1\) for \(i=1,2,3,4,5\).
\end{itemize}
Let us now examine the restricted systems for each \(\ell \in L(\kappa_0)\). 
Following the proof of \Cref{prop:hdl_projection}, for any vector \(\xi=\left(
\begin{array}{c}
\upsilon \\ \varphi 
\end{array}\right) 
\in \ker G_{0,I(\ell)}\), there is a \(0\)-cycle \(z \in \Zo_0\) with \(z \approx (\upsilon,0,\varphi)\) such that \([z] \in \Ho_0(\ell)\) and 
\([z] = \left(
\begin{array}{c}
0 \\ \varphi 
\end{array}
\right)\).
\begin{itemize}
\item \(G_{0,I(x \wedge y)}\xi = 0\) is an empty system that is trivially solved by any \(\xi \in \bbz^5\). Then, \(\Ho_0(x \wedge y) = \Ho_0 \cong \bbz^2\).
\item \(G_{0,I(x)}\xi = 0\) is a system with the third and fifth rows of \(G_{0}\):
\[
\left(
\begin{array}{rrrrr}
 0 &  0 &  1 &  0 &  0 \\
 0 &  0 &  0 &  0 &  1 \\
\end{array}
\right) \cdot 
\left(\begin{array}{r}\upsilon_1 \\\upsilon_2 \\\upsilon_3 \\\varphi_1 \\\varphi_2\end{array}\right)=
\left(\begin{array}{r}0 \\0\end{array}\right)
\]
The kernel of \(G_{0,I(x)}\) is generated by the vectors \(\{(1,0,0,0,0)',(0,1,0,0,0)',(0,0,0,1,0)'\}\). Projecting them to their last two components, we have that \(\Ho_0(x)=\langle [f^0_1]\rangle \cong \bbz\)
\item \(G_{0,I(y)}\xi = 0\) is a system with only the first, second and fourth rows of \(G_0\):
\[
\left(
\begin{array}{rrrrr}
 1 &  0 &  0 &  0 &  0 \\ 
 0 &  1 &  0 &  0 &  0 \\ 
-1 & -1 & -1 &  1 &  0 
\end{array}
\right) \cdot 
\left(\begin{array}{r} \upsilon_1  \\  \upsilon_2  \\  \upsilon_3  \\ \varphi_1  \\  \varphi_2 \end{array}\right)=
\left(\begin{array}{r}0 \\0  \\ 0 \end{array}\right)\,.
\]
The kernel of \(G_{0,I(y)}\) is generated by the vectors \(\{(0,0,1,1,0)',(0,0,0,0,1)'\}\). Projecting them to their last two components, we have that \(\Ho_0(x)=\langle [f^0_1], [f^0_2] \rangle = \Ho_0 \cong \bbz^2\).

\item \(G_{0,I(1)}\xi = 0\) is the complete system, only solvable by \(\xi = 0\) because \(G_0=M^{\Delta,H}_0\) is invertible. Thus, \(\Ho_0(1)=\{0\}\).
\end{itemize}

Now we describe the cuts of \(\eta_0\). By applying \cref{item:hdl_cut} of \Cref{thm:hdl_properties}, the cut \(\eta_0^{\geq \ell}\) is obtained by considering all subsets \(S \subseteq L(\kappa_0) = \{x \wedge y, x, y, 1\}\) such that \(\bigvee S \geq \ell\) (note that some of these subsets may be redundant). Thus:
\begin{itemize}
\item[--] For \(\ell=x \wedge y\), it suffices to consider \(S=\{x \wedge y\}\), so \(\eta_0^{\geq x \wedge y}=\Ho_0(x \wedge y)=\Ho_0 \cong \bbz^2\).
\item[--] For \(\ell=x\), it suffices to consider \(S=\{x\}\), so \(\eta_0^{\geq x}=\Ho_0(x)=\langle [f^0_1]\rangle \cong \bbz\).
\item[--] For \(\ell=y\), it suffices to consider \(S=\{y\}\), so \(\eta_0^{\geq y}=\Ho_0(y)=\Ho_0 \cong \bbz^2\).
\item[--] For \(\ell=x \vee y\), we consider \(S=\{x,y\}\), so \(\eta_0^{\geq x \vee y} = \Ho_0(x) \cap \Ho_0(y) = \Ho_0(x) = \langle [f^0_1]\rangle \cong \bbz\).
\item[--] For \(\ell=1\), the only possible subset is \(S=\{1\}\), so \(\eta_0^{\geq 1} = \Ho_0(1)=\{0\}\).
\end{itemize}
Knowing all the cuts, we can give an explicit description for \(\eta_0 \in \calfm(\Ho_0,L)\):
\[
\eta_0(\varphi_1[f^0_1]+\varphi_2[f^0_2])=
\begin{cases}
1 & \text{ if } \varphi_1=\varphi_2 = 0, \\
x \vee y & \text{ if } \varphi_2 = 0 \text{ and } \varphi_1 \neq 0, \\
y & \text{ otherwise.}
\end{cases}
\]
In particular, \(\eta_0([f^0_1])=x \vee y\) because the connected component of \(v_3\) contains points of both colors and \(\eta_0([f^0_2])=y\) because the connected component of \(v_4\) only contains a blue point.

\begin{remark}
\label{rem:counterexample1}
In \cref{item:hdl_contain} of \Cref{thm:hdl_properties} we saw that \(\Ho_0(\ell) \subset \eta_0^{\geq \ell}\) for any \(\ell \in L\). This inclusion is strict for \(\ell = x \vee y\). Indeed, \(\Ho_0(x \vee y) = \{0\}\) and \(\eta_0^{\geq x\vee y}=\langle [f^0_1]\rangle \cong \bbz\). This occurs because \([f^0_1]\) does not contain any \(0\)-chain with value above \(x \vee y\), but it contains a \(0\)-chain with value \(x\) (a red point) and a \(0\)-chain with value \(y\) (a blue point).
\end{remark}

\begin{remark}
\label{rem:counterexample2}
In \cref{item:hdl_join} of \Cref{thm:hdl_properties} we also saw that \(H_0(\bigvee S) \subset \bigcap_{\ell \in S} \Ho_0(\ell)\) for any subset \(S \subseteq L\). 
This inclusion is strict for \(S=\{x, y\}\). Indeed, \(\Ho_0(x \vee y) =\{0\}\) and \(\Ho_0(x) \cap \Ho_0(y) = \Ho_0(x)= \langle [f^0_1]\rangle\cong \bbz\).
\end{remark}


\section{Conclusions and future work}
\label{sec:conclusions}

In summary, we have defined $L$-fuzzy simplicial homology in an analogous way to
the classical definition of simplicial homology, replacing the usual crisp notions—such as simplicial complexes, modules, images and preimages under homomorphisms, and quotient modules—by their \(L\)-fuzzy counterparts. As a result, we obtain an \(L\)-fuzzy submodule \(\eta_d \in \calfm(\Ho_d,L)\), that is, a map assigning to each homology class \([h] \in \Ho_d\) a value \(\eta_d([h]) \in L\). Furthermore we have shown, both theoretically and through an explicit example, that the values \(\eta_d([h])\) and the cuts of \(\eta_d\), can be effectively computed using techniques based solely on matrix transformations and the solution of linear Diophantine systems.

A natural continuation of this work is to extend \Cref{def:eta} of $L$-fuzzy simplicial homology to relative pairs of $L$-fuzzy simplicial complexes. 
In particular, since we presented a new homology theory, it is worth showing that it satisfies the Eilenberg-Steenrod axioms~\cite{EilenbergSteenrod1952}. 
For this purpose, the simplification of the axioms~\cite{Dawson1988} is particularly useful, since it is formulated in the language of simplicial complexes, avoiding technical difficulties regarding topological spaces and homotopies. 
Here notice that the dimension axiom should be reformulated in terms of an $L$-fuzzy module.
In an analogous way to \cite{EilenbergSteenrod1952,Dawson1988}, the ultimate aim would be to show that all homology theories on $L$-fuzzy simplicial complexes are equivalent.

Another natural direction for future work is the extension of this theory to persistent homology. In practice, TDA researchers do not construct a single simplicial complex from a point cloud, 
but rather a filtration, that is, an increasing sequence of simplicial complexes (recall \Cref{def:filtration}). In these constructions, the inclusion of a potential simplex \(\sigma=\langle x_0, \dots, x_n \rangle\) in each complex of the filtration depends on a certain parameter such as proximity or density. Classical examples of this include the \v{C}ech \cite{carlsson2009topology}, Vietoris-Rips \cite{zomorodian2010fast} and Alpha filtrations \cite{giesen2006conformal}. From this perspective, we wonder which invariants can be defined for filtrations of \(L\)-fuzzy subcomplexes. In particular, we aim to develop a notion of \(L\)-fuzzy persistent homology, capable of capturing the evolution of \(L\)-fuzzy features across scales.

\appendix

\section{List of notations}
\label{sec:listofnotations}

\begin{longtable}{l|l}
\toprule
\textbf{Notation} & \textbf{Description} \\
\midrule
\endfirsthead
\toprule
\textbf{Notation} & \textbf{Description} \\
\midrule
\endhead
\midrule
\multicolumn{2}{r}{Continued on next page} \\
\midrule
\endfoot
\bottomrule
\endlastfoot
\((P,\leq)\) & Partially ordered set (poset) \\
\(\bigvee, \; \bigwedge\) & Join and meet \\
\((L,\leq)\) & Completely distributive lattice (CDL) \\
\(\fdl(x_1,\dots,x_n)\) & Free distributive lattice generated by \(x_1,\dots,x_n\) \\
\(P^{p(\ell)}\) & Filter, subset of \(P\) that satisfies the property \(p(\ell)\) \\
\(\calp(X,L)\) & Power ser of \(X\), or the set of crisp subsets of \(X\) \\
\(\calfp(X,L)\) & \(L\)-fuzzy power set of \(X\), or the set of \(L\)-fuzzy subsets of \(X\) \\
\(L(\mu)\) & Image of \(\mu: X \to L\) \\
\(S_{\ell}\) & \(L\)-fuzzy subset with constant value \(\ell\) in \(S \subset X\) and \(0\) in \(X \setminus S\) \\
\(\mathcal{D} = (X,f)\) & Chromatic dataset, with \(f: X \to C = \{c_1,\dots,c_k\}\) \\
\(\mu^{p(\ell)}\) & Filter, subset of \(X\) such that \(\mu^{-1}(\mu^{p(\ell)})=P^{p(\ell)}\) \\
\(\mu^*=\mu^{>0}\) & Support of \(\mu\) \\
\(\mu_*=\mu^{\geq 1}\) & Core of \(\mu\) \\
\(\cut(\mu): L \to \calp(X)\) & Contravariant functor such that \(\ell \mapsto \mu^{\geq \ell}\) \\
\(\bba\) & Commutative ring with \(1\) \\
\(\bbd\) & Principal Ideal Domain (PID) \\
\(\bbf\) & Field \\
\(\Mo, \No\) & \(\bba\)-modules \\
\(\calfm(\Mo,L)\) & Set of \(L\)-fuzzy submodules of \(\Mo\) \\
\((\sub(\Mo),\subset)\) & Category of crisp submodules of \(\Mo\) with morphisms restricted to inclusions \\
\(\langle S \rangle\) & Crisp submodule generated by the subset \(S\) \\
\(\langle \mu \rangle\) & \(L\)-fuzzy submodule generated by the \(L\)-fuzzy subset \(\mu\) \\
\(\Mo \cong \No\) & Isomorphic modules \\
\(\mu \cong \nu\) & Isomorphic \(L\)-fuzzy submodules \\
\(\Mo/\No\) & Quotient of modules \\
\([m]=m + N\) & Coset or class of \(m\) in the quotient module \(\Mo / \No\) \\
\(\mu/\nu\) & Quotient of \(L\)-fuzzy submodules \\
\(\rank\!\cut(\mu): L \to \bbz_{\geq 0}\) & Contravariant functor such that \(\ell \mapsto \rank(\mu^{\geq \ell})\) \\
\(\sigma = \langle v_0,\dots,v_d \rangle\) & \(d\)-simplex generated by the vertices \(v_0,\dots,v_d\) \\
\(\Delta\) & Simplicial complex \\
\(\Delta_d\) & Set of \(d\)-simplices of \(\Delta\) \\
\(\calfc(\Delta,L)\) & Set of \(L\)-fuzzy subcomplexes of \(\Delta\) \\
\((\sub(\Delta),\subset)\) & Category of crisp subcomplexes of \(\Delta\) with morphisms restricted to inclusions \\
\((\SpCpx,\subset)\) & Category of finite simplicial complexes with morphisms restricted to inclusions \\
\(\Sigma_F=\bigcup_{p \in P}F(p)\) & Simplicial complex associated to the filtration \(F: P \to \SpCpx\) \\
\(\Sigma_{M} = \bigcup_{\ell} M(\ell) \) & Simplicial complex associated to the decreasing filtration \(M: L \to \SpCpx\) \\
\(P_\uparrow\) & Set of up-sets of the poset \(P\) \\
\([v_0,\dots,v_d]\) & Oriented \(d\)-simplex associated to \(\sigma = \langle v_0,\dots,v_d \rangle\) \\
\(\Co_d\) & \(\bbd\)-module of \(d\)-chains \\
\(E^{\Delta}_d = \{\sigma^{d}_{1},\dots,\sigma^{d}_{n_d}\}\) & Basis of \(\Co_d\) given by the positively oriented \(d\)-simplices of \(\Delta\) \\
\(c^{\Delta}\) & Vector of coefficients of \(c \in \Co_d\) with respect to \(E^{\Delta}_d\) \\
\(\partial_d: \Co_d \to \Co_{d-1}\) & \(d\)-th boundary operator \\
\(\Zo_d = \ker(\partial_d)\) & Submodule of \(d\)-cycles \\
\(\Bo_d = \im(\partial_{d+1})\) & Submodule of \(d\)-boundaries \\
\(\mathcal{C}(\Delta)\) & Chain complex associated to \(\Delta\) \\
\(\Ho_d = \Zo_d / \Bo_d\) & \(d\)-homology \(\bbd\)-module \\
\(M_d \in \bbd^{n_{d-1}\times n_d}\) & Matrix representing \(\partial_d: \Co_d \to \Co_{d-1}\) respect to \(E^{\Delta}_d\) and \(E^{\Delta}_{d-1}\) \\
\(D_d \in \bbd^{n_{d-1}\times n_d}\) & Matrix representing \(\partial_d: \Co_d \to \Co_{d-1}\) respect to \(E^{H}_d\) and \(E^{H}_{d-1}\) \\
\(E^{H}_d\) & Basis of \(\Co_d\) given after reducing \(M_d\) into \(D_d\) \\
\(c^{H}\) & Vector of coefficients of \(c \in \Co_d\) respect to \(E^{H}_d\) \\
\(\mhd_d, \mdh_d\) & Change-of-basis matrices between \(E^{H}_d\) and \(E^{\Delta}_d\) \\
\(u^d_1,\ldots,u^d_{n_U}\) & Generators of \(E^{H}_d\) in group U \\
\(t^d_1,\ldots,t^d_{n_T}\) & Generators of \(E^{H}_d\) in group T \\
\(a^d_1,\ldots,a^d_{n_T}\) & Torsion coefficients of \(\Ho_d\) \\
\(r^d_1,\ldots,r^d_{n_R}\) & Generators of \(E^{H}_d\) in group R \\
\(f^d_1,\ldots,f^d_{n_F}\) & Generators of \(E^{H}_d\) in group F \\
\(U_d, T_d, R_d, F_d\) & Submatrix of \(\mdh_d\) corresponding to generators of \(E^{H}_d\) in groups U, T, R, F  \\
\(A_d = \operatorname{diag}(a^d_1,\dots,a^d_{n_T})\) & Diagonal matrix with the torsion coefficients of \(\Ho_d\) \\
\([h]=\left(\begin{array}{c} \alpha\\ \varphi \end{array} \right)\) & Compact coordinates of \([h] \in \Ho_d\) \\
\(c \approx (\upsilon_c,\tau_c,\rho_c,\varphi_c)\) & Compact coordinates of \(c \in \Co\) respect to \(E^{H}_d\) \\
\(z \approx (\upsilon_z,\tau_z,\varphi_z)\) & Compact coordinates of \(z \in \Zo\) respect to \(E^{H}_d\) \\
\(b \approx (\upsilon_b,\tau_b)\) & Compact coordinates of \(b \in \Bo\) respect to \(E^{H}_d\) \\
\(z \approx h + (\upsilon_b,\tau_b)\) & Compact coordinates of \(z \in [h]\) if \(z = h+b\) for some \(b \in \Bo\) \\
\(\pi^{d}_i: \bbd \to \bbd/(a^d_i)\) & Natural projection associated to the torsion coefficient \(a^d_i\) \\
\([z]=\left(\begin{array}{c} \pi^d(\tau_z)\\ \varphi_z \end{array} \right)\) & Compact coordinates of \([z] \in \Ho_d\) if \(z \approx (\upsilon_z,\tau_z,\varphi_z)\) \\
\(\delta_d \in \calfp(\Co_d,L)\) & \(L\)-fuzzy subset of \(d\)-simplices \\
\(\kappa_d \in \calfm(\Co_d,L)\) & \(L\)-fuzzy submodule of \(d\)-chains \\
\(\zeta_d \in \calfm(\Co_d,L)\) & \(L\)-fuzzy submodule of \(d\)-cycles \\
\(\beta_d \in \calfm(\Co_d,L)\) & \(L\)-fuzzy submodule of \(d\)-boundaries \\
\(\eta_d = \zeta_d / \beta_d\) & \(L\)-fuzzy \(d\)-homology \\
\(M_I\) & Submatrix of \(M\) with the rows indexed by \(I\) \\
\(S(c,I)\) & System of linear diophantine equations given by \((U_d \mid T_d A_d)_I \left(\begin{array}{c} \upsilon\\ \tau \end{array} \right) = -c^{\Delta}_I\) \\
\(I(\ell) = \big\{i \mid \sigma^{d}_{i} \in \mu^{\not\geq \ell}\big\}\) & Set of indices associated to \(\ell \in L\). \\
\(\Ho_d(\ell)\) & Crisp submodule of \(\Ho_d\) given by \(\big\{[h] \in \Ho_d \mid S(h,I(\ell)) \text{ is solvable} \big\}\) \\
\(\solv(\eta_d): L \to \sub(\Ho_d)\) & Contravariant functor given by \(\ell \mapsto \Ho_d(\ell)\) \\
\(G_d\) & Submatrix of \(\mdh_d\) given by \((U_d \mid T_d \mid F_d)\) \\
\end{longtable}



\section*{Acknowledgments}

\noindent Javier Perera-Lago was funded by a \textit{Junta de Andalucía} predoctoral grant with reference:
``DGP\_PRED\_ 2024\_02465''.

\bibliographystyle{plain}
\bibliography{bibliography}







\end{document}